\documentclass{article}

\usepackage{amsthm}
\usepackage{amsmath}
\usepackage{amssymb}
\usepackage{amsfonts}
\usepackage{framed}
\usepackage{xcolor}
\usepackage{fullpage}
\usepackage{hyperref}
\hypersetup{
    colorlinks=true,
    linkcolor=blue!60!black,
    citecolor=blue!60!black,
    urlcolor=blue!60!black,
    filecolor=blue!60!black,
    linktoc=all
}
\usepackage{algorithm, algorithmic}
\usepackage{graphicx, caption, subcaption}
\usepackage[authoryear,round,sort&compress]{natbib}
\usepackage{booktabs}
\usepackage{makecell}

\newcommand*{\dt}[1]{\dot{#1}}

\DeclareMathOperator{\rank}{rank}
\DeclareMathOperator{\RGS}{RGS}

\newcommand{\norm}[1]{\left\lVert#1\right\rVert}
\newcommand{\scal}[1]{\left\langle #1 \right\rangle}

\newcommand{\R}{\mathbb{R}}
\newcommand{\Rmn}{\mathbb{R}^{m \times n}}
\newcommand{\Rnr}{\mathbb{R}^{n \times r}}
\newcommand{\Rmr}{\mathbb{R}^{m \times r}}
\newcommand{\Rrr}{\mathbb{R}^{r \times r}}

\newcommand{\Mr}{\mathcal{M}_r}
\newcommand{\T}{\mathcal T}
\newcommand{\proj}[2]{\mathcal{P}_{#1} \left[ #2 \right]}
\newcommand{\projperp}[2]{\mathcal{P}_{#1}^{\perp} \left[ #2 \right]}
\newcommand{\Range}{\mathrm{Range}}
\renewcommand{\Im}{\mathrm{Im}}

\newtheorem{theorem}{\bf Theorem}[section]
\newtheorem{definition}[theorem]{\bf Definition}
\newtheorem{lemma}[theorem]{\bf Lemma}
\newtheorem{remark}[theorem]{\bf Remark}
\newtheorem{corollary}[theorem]{\bf Corollary}
\newtheorem{proposition}[theorem]{\bf Proposition}
\newtheorem{assumption}{Assumption}

\title{Sketch low-rank dynamics: orthogonal vs.\ oblique projections}
\author{Benjamin Carrel \thanks{The work of this author was performed while at PSI Center for Scientific Computing, Theory and Data, Paul Scherrer Institute, Villigen PSI, Switzerland}, Laura Grigori\thanks{PSI Center for Scientific Computing, Theory and Data, Paul Scherrer Institute, Villigen PSI, and Institute of Mathematics, EPFL, Switzerland}}
\date{July 2026}

\begin{document}

\maketitle

\begin{abstract}
We study how sketching techniques from randomized numerical linear algebra can be incorporated into the dynamical low-rank approximation (DLRA) of large-scale matrix differential equations.
A natural approach is to sketch the Galerkin condition that defines the DLRA, which leads to an oblique tangent space projection. We show that this oblique projection approximately reproduces the standard DLRA only under restrictive conditions on the vector field, and that it fails on problems with a large perpendicular residual.
As an alternative, we propose an orthogonal sketch DLRA that evolves sketch-orthogonal bases while using standard orthogonal projections for the dynamics. This approach preserves the geometric structure of the classical DLRA and is numerically stable. The computational advantage of randomized Gram--Schmidt over Householder QR lies in fewer global synchronizations on a row-distributed basis, at a comparable flop count; when the basis is well conditioned, randomized Gram--Schmidt can be replaced by randomized Cholesky QR, which additionally shifts the basis update from BLAS-2 to BLAS-3 kernels, making it well-suited to modern accelerators.
We derive sketch versions of the projector-splitting and BUG integrators, and demonstrate the approach on the Allen--Cahn, Fokker--Planck, and Vlasov--Poisson equations.
\end{abstract}

\section{Introduction} \label{sec:introduction}

Large-scale matrix differential equations arise in a variety of scientific computing applications. In kinetic theory, the Vlasov--Poisson equation describes the evolution of particle distributions in collisionless plasmas. In materials science, the Allen--Cahn equation models phase separation processes. In stochastic analysis, the Fokker--Planck equation governs the evolution of probability density functions. When discretized on tensor-product grids, these equations lead to matrix differential equations of the form
\begin{equation} \label{eq:fom_intro}
\dt{A}(t) = F(t, A(t)), \quad A(0) = A_0 \in \Rmn,
\end{equation}
where $m$ and $n$ can be very large. We refer to \eqref{eq:fom_intro} as the full order model (FOM). In the rest of the paper, we will consider autonomous problems, without loss of generality, in order to keep the notation light.

In the last decade, various techniques for compressing the dynamics have been proposed. One of the most popular and well-studied techniques is the dynamical low-rank approximation (DLRA)~\citep{koch2007dynamical}, which has since been applied to kinetic equations~\citep{einkemmer2019quasi, einkemmer2021mass}, weakly compressible flows~\citep{einkemmer2019low}, and high-dimensional nonlinear PDEs~\citep{dektor2021rank}. The key idea is to evolve the lower-dimensional dynamics by projecting the vector field onto the tangent space of the manifold $\Mr$ of fixed rank-$r$ matrices. This yields a system of coupled differential equations for the low-rank factors, reducing the complexity from $O(mn)$ to $O((m+n)r)$. The DLRA is defined by the Dirac--Frenkel variational principle:
\begin{equation}
    \dt{Y}(t) = \mathcal P_{Y(t)}F(Y(t)), \quad Y(0) = Y_0 \in \Mr,
\end{equation}
where $\mathcal P_Y$ denotes the $\ell_2$-orthogonal projection onto the tangent space $\T_Y \Mr$.

Practical numerical integrators for the DLRA include the projector-splitting integrator~\citep{lubich2014projector} and the BUG integrator~\citep{ceruti2022unconventional}, together with its higher-order and rank-adaptive variants~\citep{ceruti2022rank, ceruti2024robust} and related projection methods~\citep{kieri2019projection, carrel2023projected}, as well as interpolatory low-rank integrators~\citep{carrel2025interpolatory}. These methods always require QR decompositions at each time step in order to orthogonalize the updated bases. In the large-scale setting on distributed hardware, the Householder QR of $K \in \Rmr$ suffers from a synchronization bottleneck: it performs $r$ sequential panel factorizations, each requiring a global reduction on the row-distributed factor~\citep{demmel2012communication}; the panel factorizations themselves are BLAS-2 operations. Randomized numerical linear algebra (RNLA)~\citep{halko2011finding, martinsson2020randomized, murray2023randomized, nakatsukasa2024fast} provides a natural alternative. These sketch-based orthogonalizations replace the Householder QR by working on a small sketch $\Theta K \in \R^{\ell \times r}$ of the tall factor $K \in \Rmr$ (with $\Theta \in \R^{\ell \times m}$). In its simplest form, randomized Cholesky QR computes the QR decomposition of $\Theta K$ and uses the resulting triangular factor to orthogonalize $K$, collapsing the whole orthogonalization to a single BLAS-3 matrix multiply and a single global reduction. The more robust randomized Gram--Schmidt (RGS) builds the basis column by column instead; at the typical sketch size $\ell = O(r)$ its leading-order flop count $O(m \ell r)$ matches Householder QR, while it requires far fewer global synchronizations --- the source of the speedup on modern parallel hardware (Section~\ref{subsec:cost}).

Beyond the orthogonalization step, randomization has also been used directly inside the time integration of low-rank matrix ODEs, most notably through the randomized Runge--Kutta methods of \citet{lam2025randomized} --- a direction complementary to ours, which keeps the time integrator classical and randomizes only the basis update. Closely related are the dynamical randomized methods of \citet{carrel2025randomized}, which build the low-rank integrator itself from randomized linear algebra --- randomizing the range-finding and projection steps rather than only the orthogonalization --- and report accurate, low-variance, structure-preserving behavior on stiff problems, including the conservation of physical invariants on Vlasov--Poisson. The sketch orthogonalization developed here can serve as the basis update within those methods as well.

There are two natural ways to incorporate sketching into DLRA:
\begin{itemize}
\item \textbf{Oblique approach:} sketch the Galerkin condition that defines the DLRA. This leads to an oblique tangent space projection involving the sketch matrices $\Theta$ and $\Omega$. The resulting dynamics are \emph{different} from the standard DLRA.
\item \textbf{Orthogonal approach:} evolve sketch-orthogonal bases (e.g.\ via RGS) but use standard orthogonal projections for the dynamics. This preserves the geometric structure of the classical DLRA. The sketch only affects the runtime of the orthogonalization process, not the dynamics.
\end{itemize}

The key mathematical insight of this paper is that the oblique approach changes the dynamics and leads to errors on problems with a large perpendicular residual (see Definition~\ref{def:lrc_approx_rk}), while the orthogonal approach preserves the dynamics exactly. Table~\ref{tab:intro_comparison} summarizes the qualitative differences; a detailed comparison is given in Table~\ref{tab:comparison}.
This is illustrated in Figure~\ref{fig:motivating}, which shows the energy over time of a two-stream instability simulation for both approaches. The oblique sketch DLRA exhibits energy drift, while the orthogonal sketch DLRA matches the standard DLRA with a near-constant energy in the nonlinear regime.

\begin{figure}[htbp]
    \centering
    \includegraphics[width=0.8\textwidth]{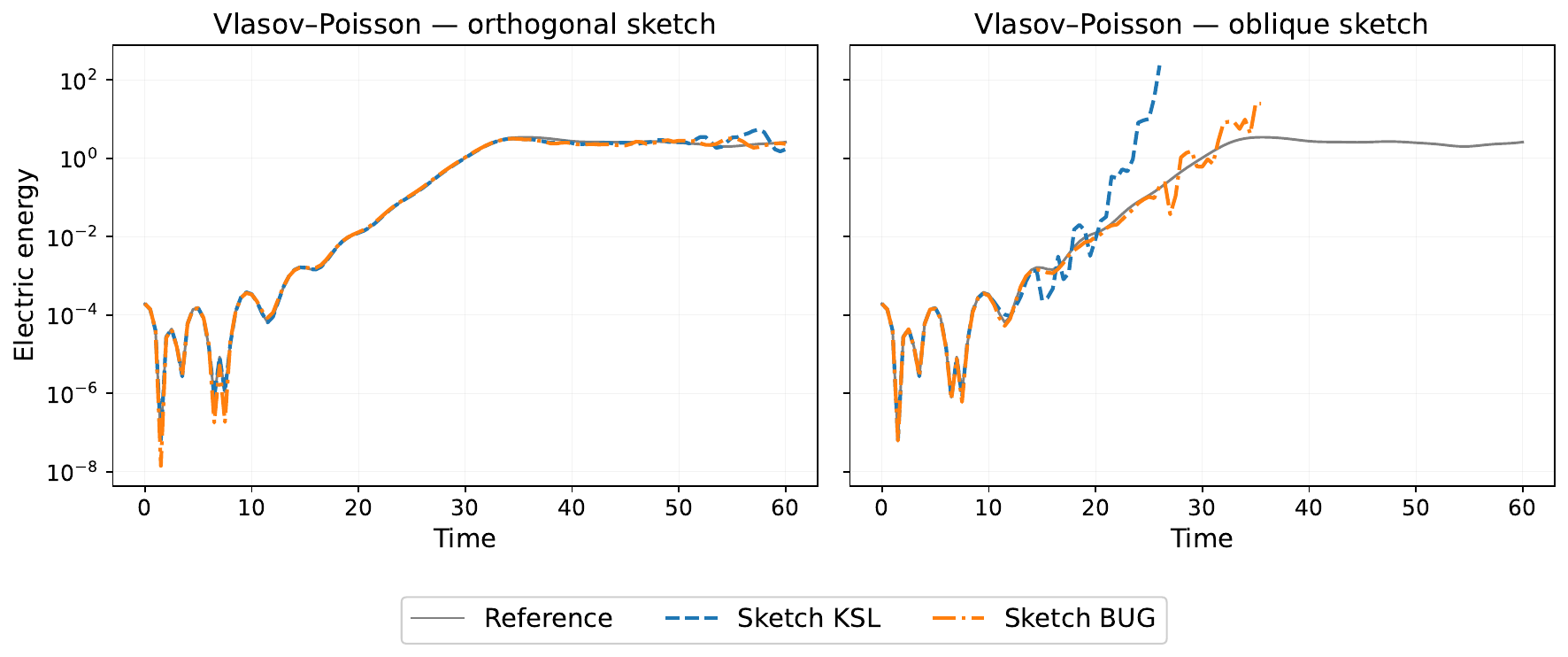}
    \caption{Motivating example: electric energy on the two-stream instability (Vlasov--Poisson equation~\eqref{eq:vlasov_poisson}) with $64$ grid points and rank-$8$ approximation. The orthogonal sketch DLRA (left figure) reproduces the reference electric-energy trajectory, while the oblique sketch DLRA (right figure) drifts away from it.}
    \label{fig:motivating}
\end{figure}

\begin{table}[htbp]
\centering
\begin{tabular}{lcc}
\toprule
 & \textbf{Oblique sketch DLRA} & \textbf{Orthogonal sketch DLRA} \\
\midrule
\textbf{Dynamics} & Perturbed & Same as DLRA \\
\textbf{Equivalent to DLRA} & Approximately, on LRC problems & Always \\
\textbf{\makecell[l]{Large perpendicular residual \\ (e.g.\ Vlasov--Poisson)}} & \makecell{Drifts from DLRA \\ (e.g.\ energy drift)} & Matches DLRA \\
\bottomrule
\end{tabular}
\caption{Qualitative comparison of the two approaches. See Table~\ref{tab:comparison} for a detailed version.}
\label{tab:intro_comparison}
\end{table}

\subsection*{Contributions and outline}

The main contributions of this paper are:
\begin{enumerate}
    \item We propose an orthogonal sketch DLRA framework that evolves sketch-orthogonal bases $(P, W)$ while using standard orthogonal projections. This is connected to the classical DLRA via a bijective map $\varphi$ that relates sketch-orthogonal and orthogonal factorizations.
    \item We show that sketching the Galerkin condition leads to an oblique projection that is equivalent to the classical DLRA \emph{if and only if} the sketches make the perpendicular component $\mathcal P_Y^\perp F(Y)$ orthogonal to the tangent space (Proposition~\ref{prop:oblique_equivalence}). This condition holds up to the sketching error on low-rank compatible problems (Section~\ref{subsec:equivalence}), while on problems with a large perpendicular residual (e.g.\ Vlasov--Poisson) the oblique approach fails in general.
    \item We derive sketch versions of the projector-splitting and BUG integrators, replacing Householder QR with RGS.
    \item We demonstrate the approach on the Allen--Cahn, Fokker--Planck, and Vlasov--Poisson equations. The orthogonal sketch DLRA matches the accuracy of the standard DLRA, while the oblique approach loses physical invariants (e.g.\ the electric energy) as soon as the perpendicular residual becomes large.
\end{enumerate}

Preliminaries are collected in Section~\ref{sec:preliminaries}; the orthogonal and oblique constructions are developed in Sections~\ref{sec:orthogonal} and~\ref{sec:oblique}; sketch integrators and their cost are derived in Section~\ref{sec:integrators}; numerical experiments are reported in Section~\ref{sec:experiments}.

\section{Preliminaries} \label{sec:preliminaries}

\subsection{Oblivious subspace embeddings} \label{subsec:ose}

We recall the definition of an oblivious subspace embedding (OSE), introduced by \citet{sarlos2006improved} in a sketching context and subsequently surveyed in~\citep{halko2011finding, woodruff2014sketching, martinsson2020randomized, murray2023randomized}.

\begin{definition}[Oblivious subspace embedding] \label{def:ose}
  A distribution over random matrices $\Omega \in \R^{\ell \times m}$ is an oblivious subspace embedding with parameters $\mathrm{OSE}(k, \epsilon, \delta)$ if, for every $k$-dimensional subspace $\mathcal V \subset \R^m$,
  $$ (1 - \epsilon) \norm{x}_2^2 \;\leq\; \norm{\Omega x}_2^2 \;\leq\; (1 + \epsilon) \norm{x}_2^2 \qquad \text{for all } x \in \mathcal V, $$
  holds with probability at least $1 - \delta$ over the draw of $\Omega$. Equivalently, 
  $$ | \scal{\Omega x, \Omega y} - \scal{x, y} | \leq \epsilon \norm{x}_2 \norm{y}_2 \qquad \text{for all } x, y \in \mathcal V. $$
\end{definition}

Typical choices for oblivious subspace embeddings include Gaussian random matrices, the subsampled randomized Hadamard transform (SRHT)~\citep{ailon2006approximate}, and sparse sign embeddings~\citep{tropp2017practical}. The sampling dimension $\ell$ depends on the target subspace dimension: for the rank-$r$ subspaces that arise in this paper, a practical choice for large-scale problems is $\ell \in [2r, 4r]$~\citep{balabanov2022randomized}, while for the moderate sizes of Section~\ref{sec:experiments} we find that $\ell = 2r$ is already sufficient. In both regimes $\ell \ll m, n$. The OSE property extends from individual vectors to tall matrices, as recorded in the next proposition.

\begin{proposition}[Tall matrix embedding~\citep{sarlos2006improved}] \label{prop:tall_embedding}
  Let $\Omega$ be an $\mathrm{OSE}(2r, \epsilon, \delta)$ and let $U, V \in \Rmr$ be two fixed tall matrices ($r \ll m$), independent of the draw of $\Omega$. Then, with probability at least $1 - \delta$,
  $$| \scal{\Omega U, \Omega V} - \scal{U,V} | \leq \epsilon \norm{U}_F \norm{V}_F.$$
  Equivalently, we also have
  $$| \scal{U^T \Omega^T, V^T \Omega^T} - \scal{U^T,V^T} | \leq \epsilon \norm{U}_F \norm{V}_F.$$
  In particular, when the matrices $U$ and $V$ have orthonormal columns, we have
  $$| \norm{\Omega U}_F^2 - \norm{U}_F^2 | \leq \epsilon \cdot r, \quad | \norm{V^T \Omega^T}_F^2 - \norm{V}_F^2 | \leq \epsilon \cdot r.$$
\end{proposition}

The proposition follows by applying the OSE property in its equivalent inner-product form (Definition~\ref{def:ose}) to the fixed subspace $\Range([U \;\, V]) \subset \R^m$, whose dimension is at most $2r$, and summing over the $r$ column pairs.

The two-sided pairings that appear in Sections~\ref{sec:oblique} and~\ref{subsec:equivalence} are controlled by the following corollary, obtained by applying Proposition~\ref{prop:tall_embedding} on each side and collecting the cross terms.

\begin{corollary}[Two-sided embedding] \label{cor:two_sided}
Let $\Theta \in \R^{\ell \times m}$ and $\Omega \in \R^{\ell \times n}$ be $\mathrm{OSE}(d, \epsilon, \delta)$, and let $Z, Z' \in \Rmn$ be fixed (independently of $\Theta$ and $\Omega$) with a combined column space and a combined row space of dimension at most $d$. Then, with probability at least $1 - 2\delta$,
$$\bigl| \scal{\Theta Z \Omega^T, \, \Theta Z' \Omega^T} - \scal{Z, Z'} \bigr| \;\leq\; (2\epsilon + \epsilon^2)\, \norm{Z}_F\, \norm{Z'}_F,$$
and in particular $\norm{\Theta Z \Omega^T}_F^2 \geq (1 - \epsilon)^2 \norm{Z}_F^2$.
\end{corollary}

\subsection{Dynamical low-rank approximation} \label{subsec:dlra}

Recall the full order model (FOM) equation $\dt A(t) = F(A(t))$ with initial condition $A(0) = A_0 \in \Rmn$, introduced in~\eqref{eq:fom_intro}. The dynamical low-rank approximation (DLRA) approximates the FOM trajectory by a curve on the manifold of fixed-rank-$r$ matrices $\Mr = \{Y \in \Rmn : \rank(Y) = r\}$, parameterized as $Y = U \Sigma V^T$ with $U \in \mathrm{St}(m,r)$, $V \in \mathrm{St}(n,r)$, and $\Sigma \in \mathrm{GL}(r)$. Its tangent space at $Y$ is the set
$$\T_Y \Mr = \{\delta U \Sigma V^T + U \delta \Sigma V^T + U \Sigma \delta V^T \mid \delta U \in \Rmr,\ \delta \Sigma \in \Rrr,\ \delta V \in \Rnr,\ \delta U^T U = 0,\ \delta V^T V = 0\},$$
where $\delta U^T U = 0$ and $\delta V^T V = 0$ are gauge conditions ensuring a unique representation of tangent vectors (see~\citep{absil2008optimization, vandereycken2013low} for the geometric background of Stiefel and low-rank manifolds). The orthogonal projection onto $\T_Y \Mr$ is given by
\begin{equation} \label{eq:tangent_proj}
\mathcal P_Y Z = UU^T Z + Z VV^T - UU^T Z VV^T.
\end{equation}

The DLRA trajectory $Y(t)$ is defined by the Galerkin condition
\begin{equation} \label{eq:galerkin}
\scal{\dt{Y}(t) - F(Y(t)), \, \delta Y} = 0 \qquad \forall \delta Y \in \mathcal T_{Y(t)} \mathcal M_r.
\end{equation}
Decomposing the vector field as $F(Y) = \mathcal P_Y F(Y) + \mathcal P_Y^\perp F(Y)$, the Galerkin condition automatically annihilates the perpendicular component against any tangent direction, since $\scal{\mathcal P_Y^\perp F(Y),\, \delta Y} = 0$ for every $\delta Y \in \T_Y \Mr$. The size of this perpendicular component is what governs the modeling error of the DLRA.
Using the projector $\mathcal P_Y$ above, the Galerkin condition is equivalent to the projected matrix differential equation
\begin{equation} \label{eq:dlra}
\dt{Y}(t) = \mathcal P_{Y(t)} F(Y(t)), \quad Y(0) = Y_0 \in \Mr,
\end{equation}
which is the standard DLRA equation found in the literature~\citep{koch2007dynamical}. Throughout, the vector field in the low-rank dynamics is evaluated at the current approximation $Y$ --- we write $F(Y)$, or $F(U \Sigma V^T)$ (resp.\ $F(P S W^T)$) in factored form --- and we reserve $F(A)$ for the full-order model~\eqref{eq:fom_intro}; when $Y \approx A$ the two fields nearly coincide.

From that description of the tangent space and with the Galerkin condition, it is straightforward to derive the three coupled differential equations
\begin{equation} \label{eq:coupled_equations}
    \left\{
    \begin{aligned}
        &\dt{U} = (I - UU^T) F(U \Sigma V^T) V \Sigma^{-1}, \\
        &\dt{\Sigma} = U^T F(U \Sigma V^T) V, \\
        &\dt{V} = (I - VV^T) F(U \Sigma V^T)^T U \Sigma^{-T}.
    \end{aligned}
    \right.
\end{equation}

As shown in~\citep{koch2007dynamical}, solving these three coupled equations is ill-conditioned when $\Sigma$ has small singular values, because of the $\Sigma^{-1}$ factor. This was a major numerical issue for DLRA and motivated the development of the projector-splitting and BUG integrators~\citep{lubich2014projector, ceruti2022unconventional}, which avoid inverting $\Sigma$ by splitting the dynamics into well-conditioned substeps.

A central notion for the upcoming analysis of the oblique sketch DLRA is \emph{low-rank compatibility}, which classifies how much of the vector field $F$ escapes the tangent space of $\Mr$. We use $\eta$ for the low-rank compatibility tolerance throughout, to avoid collision with the OSE parameter $\epsilon$ of Section~\ref{subsec:ose}.
\begin{definition}[Approximate low-rank compatibility] \label{def:lrc_approx_rk}
A problem $\dt{A} = F(A)$ is \emph{$(\eta, r, k)$-low-rank compatible} if
$$\norm{\projperp{Y}{F(Y)} - \left[\projperp{Y}{F(Y)}\right]_k}_F \leq \eta \qquad \text{for every } Y \in \Mr \cap \mathcal N,$$
where $[\cdot]_k$ denotes the best rank-$k$ approximation in the Frobenius norm, and $\mathcal N$ denotes a certain neighborhood of the solution to the problem.
\end{definition}

The special case $\eta = 0$ yields \emph{exact $(r,k)$-low-rank compatibility} (the perpendicular component has rank at most $k$), while the case $k = 0$ recovers the norm-based condition of \citet{kieri2016discretized}.

\subsection{Sketch-orthogonality and randomized Gram--Schmidt} \label{subsec:sketch_ortho}

We now introduce the notion of sketch-orthogonality, which will be central to our approach.
Consider two oblivious subspace embeddings $\Theta \in \R^{\ell \times m}$ and $\Omega \in \R^{\ell \times n}$ with sampling parameter $\ell > 0$.

\begin{definition}[Sketch Stiefel manifold]
The sketch Stiefel manifolds are defined as
$$\mathrm{St}_{\Theta}(m, r) = \left\{ P \in \R^{m \times r} \mid (\Theta P)^T (\Theta P)= I_r \right\},$$
$$\mathrm{St}_{\Omega}(n, r) = \left\{ W \in \R^{n \times r} \mid (\Omega W)^T (\Omega W) = I_r \right\}.$$
\end{definition}

A matrix $P \in \mathrm{St}_\Theta(m, r)$ is said to have \emph{sketch-orthogonal} columns with respect to $\Theta$. Note that $P$ itself does not have orthogonal columns in general; only its sketch $\Theta P \in \R^{\ell \times r}$ does.

\begin{assumption}[Non-degenerate sketching] \label{ass:embedding}
Throughout, the sketches $\Theta$ and $\Omega$ are injective on the column and row spaces of every rank-$r$ iterate $Y = P S W^T$ that arises --- equivalently, $\Theta P$ and $\Omega W$ have full column rank $r$. This is exactly the embedding property of an oblivious subspace embedding: an $\mathrm{OSE}(r, \epsilon, \delta)$ is an $\epsilon$-embedding of any fixed rank-$r$ subspace, hence injective on it, with probability at least $1 - \delta$, and we condition on this favorable event. We also take $\Theta$ and $\Omega$ to have full row rank $\ell$ (with $\ell \le m, n$), which holds almost surely for the Gaussian and SRHT sketches used here and ensures that the map $\delta Y \mapsto \Theta\, \delta Y\, \Omega^T$ sends $\T_Y \Mr^{\Theta, \Omega}$ onto the tangent space at $\Theta Y \Omega^T$ (used in the proof of Lemma~\ref{lem:oblique_coupled}).
\end{assumption}

\begin{remark}[Notation for pseudo-inverses] \label{rem:pinv_notation}
Throughout this paper, $(\Theta P)^+$ denotes the Moore--Penrose pseudo-inverse (tall matrix with full column rank: a \emph{left} pseudo-inverse). Its transpose will be denoted $(\Theta P)^{+T} := ((\Theta P)^+)^T = ((\Theta P)^T)^+$. The same applies to $(\Omega W)^+$. Under sketch-orthogonality the columns of $\Theta P$ are orthonormal, so $(\Theta P)^+ = (\Theta P)^T = P^T \Theta^T$ and $(\Theta P)^{+T} = \Theta P$; symmetrically $(\Omega W)^+ = W^T \Omega^T$. For the matrices $P^T P \in \R^{r \times r}$ and $W^T W \in \R^{r \times r}$ that appear in the orthogonal formulas, we keep the ordinary inverse $(P^T P)^{-1}$, $(W^T W)^{-1}$, since these matrices are square and invertible (under Assumption~\ref{ass:embedding}).
\end{remark}

Given $K \in \Rmr$ with columns $k_1, \dots, k_r$, the simplest approach to produce a sketch-orthogonal basis $P$ is to sketch $\Theta K \in \R^{\ell \times r}$, compute its thin QR factorization $\Theta K = Q R$, and set $P = K R^{-1}$. Then $\Theta P = Q$ has orthonormal columns, so $P \in \mathrm{St}_\Theta(m, r)$. This process is called \emph{randomized Cholesky QR} (rCholQR)~\citep{balabanov2022cholesky, balabanov2025randomized}. However, it can be unstable when $K$ is ill conditioned or breaks down if $\Theta K$ is rank-deficient, since $R$ would then be singular. The \emph{randomized Gram--Schmidt} (RGS) process described in~\citep{balabanov2022randomized} avoids this issue by progressively building the basis $P$, as in the Gram--Schmidt process, while replacing orthogonal projections with oblique projections to produce a $P$ whose columns are orthonormal with respect to the sketched inner product $\scal{\Theta \cdot, \Theta \cdot}$. In more detail, RGS builds $P = [p_1, \dots, p_r]$ and its sketch $\Theta P$ column by column as follows: at step $j$, given the columns built previously $P_{<j} = [p_1, \dots, p_{j-1}]$ and their sketch $\Theta P_{<j}$,
\begin{itemize}
    \item sketch the new column, $s_j = \Theta k_j$;
    \item compute the projection coefficients in the sketch space, $r_{<j,j} = (\Theta P_{<j})^T s_j$;
    \item form the residuals,
    $\tilde s_j = s_j - (\Theta P_{<j})\, r_{<j,j}$ and $\tilde p_j = k_j - P_{<j}\, r_{<j,j}$;
    \item normalize using the sketched norm, $r_{jj} = \norm{\tilde s_j}_2$, $\Theta p_j = \tilde s_j / r_{jj}$, $p_j = \tilde p_j / r_{jj}$.
\end{itemize}
By construction $\Theta P$ has orthonormal columns.
In our experiments, RGS keeps the process numerically stable even when $\Theta K$ is close to rank-deficient, and we apply a single rCholQR refinement to the resulting basis to enforce sketch-orthogonality up to machine precision~\citep{balabanov2022randomized,dedamas2025arnoldi}.

For a structured random sketch such as the subsampled randomized Hadamard transform, computing all $r$ sketches $s_j = \Theta k_j$ costs $O(m r \log \ell)$ in total (or $O(m \ell r)$ for a dense Gaussian sketch). The remaining per-step work is dominated by the long-vector update $\tilde p_j = k_j - P_{<j}\, r_{<j,j}$, costing $2m(j-1)$ flops at step $j$; the sketch-side operations are $O(\ell r)$ per step. With the typical sketch size $\ell = O(r)$, summing over $j = 1, \dots, r$ and ignoring lower-order terms yields a total cost of $mr^2$ flops for RGS. In addition, RGS can exploit mixed precision. It is shown in \citep{balabanov2022randomized} that the operation $P_{<j}\, r_{<j,j}$ can be performed with lower precision while still obtaining a well-conditioned basis $P$. 
We analyze the cost in more detail in Section~\ref{subsec:cost}, and refer to~\citep{balabanov2022randomized} for the cost and stability analysis of RGS. 

\section{Orthogonal sketch DLRA} \label{sec:orthogonal}

In this section, we propose the orthogonal sketch DLRA: we evolve sketch-orthogonal bases but keep \emph{standard orthogonal} projections in the dynamics. The sketch is used only for the orthogonalization step (replacing Householder QR with RGS), not for the projection itself. As we shall see in Section~\ref{sec:oblique}, this is in contrast with the more natural approach of sketching the Galerkin condition, which leads to an oblique projection that can fail with a large perpendicular residual.

\subsection{Sketch-orthogonal manifold and bijective map} \label{subsec:bijective_map}

We define the set of fixed rank $r$ matrices with sketch-orthogonality conditions (w.r.t.\ $\Theta$ and $\Omega$) as
$$\mathcal M_r^{\Theta, \Omega} = \left\{ P S W^T \mid S \in \mathrm{GL}(r), P \in \mathrm{St}_{\Theta}(m, r), W \in \mathrm{St}_{\Omega}(n, r) \right\},$$
and its tangent space at $Y = PSW^T \in \Mr^{\Theta, \Omega}$ as
\begin{equation*}
\mathcal T_Y \Mr^{\Theta, \Omega} = \left\{ \delta P S W^T + P \delta S W^T + P S \delta W^T \;\middle|\;
\begin{aligned}
& \delta S \in \R^{r \times r},\ \delta P \in \R^{m \times r},\ \delta W \in \R^{n \times r} \\
& \delta P^T \Theta^T \Theta P =0, \ \delta W^T \Omega^T \Omega W =0
\end{aligned}
\right\}.
\end{equation*}
The special case $\Theta = I_m$ and $\Omega = I_n$ recover the standard manifold $\Mr$ and its tangent space. For a rectangular sketch ($\ell < m, n$), the identification with $\Mr$ is pointwise rather than global: the inclusion $\Mr^{\Theta, \Omega} \subseteq \Mr$ always holds, and a matrix $Y \in \Mr$ admits a sketch-orthogonal factorization exactly when the sketches are injective on its column and row spaces, that is,
$$\Mr^{\Theta, \Omega} \;=\; \left\{ Y \in \Mr \;\middle|\; \Theta \text{ injective on } \Range(Y) \text{ and } \Omega \text{ injective on } \Range(Y^T) \right\}.$$
By Assumption~\ref{ass:embedding}, every iterate arising in this paper belongs to $\Mr^{\Theta, \Omega}$, and at every such point, the tangent spaces coincide as subsets of $\Rmn$:
$$\T_Y \Mr^{\Theta, \Omega} = \T_Y \Mr \qquad \text{for all } Y \in \Mr^{\Theta, \Omega}.$$
The sketch-orthogonality conditions only affect the parameterization, not the underlying geometric objects. This equivalence is made precise by the following map $\varphi$:
\begin{align*}
\varphi : \quad &\mathrm{St}_{\Theta}(m, r) \times \mathrm{GL}(r) \times \mathrm{St}_{\Omega}(n, r) \rightarrow \mathrm{St}(m, r) \times \mathrm{GL}(r) \times \mathrm{St}(n, r) \\
&\left( P,\ S,\ W \right) \mapsto \left( P (P^T P)^{-1/2},\ (P^T P)^{1/2} S (W^T W)^{1/2},\ W (W^T W)^{-1/2} \right),
\end{align*}
which is a bijection onto its image --- the orthogonal triples whose column and row spaces are embedded by $\Theta$ and $\Omega$, with an inverse map
\begin{align*}
\varphi^{-1} : \quad &\operatorname{Im}\varphi  \rightarrow \mathrm{St}_{\Theta}(m, r) \times \mathrm{GL}(r) \times \mathrm{St}_{\Omega}(n, r) \\
&\left( U,\ \Sigma,\ V \right) \mapsto \left( U (U_{\Theta}^T U_{\Theta})^{-1/2},\ (U_{\Theta}^T U_{\Theta})^{1/2} \Sigma (V_{\Omega}^T V_{\Omega})^{1/2},\ V (V_{\Omega}^T V_{\Omega})^{-1/2} \right),
\end{align*}
where $U_{\Theta} = \Theta U$, $V_{\Omega} = \Omega V$, and $(\cdot)^{1/2}$ denotes the symmetric positive square root of a symmetric positive definite matrix. The domain $\operatorname{Im}\varphi = \{(U, \Sigma, V) \in \mathrm{St}(m, r) \times \mathrm{GL}(r) \times \mathrm{St}(n, r) : \Theta U \text{ and } \Omega V \text{ have full column rank}\}$ is exactly the set of orthogonal triples whose column and row spaces are embedded by $\Theta$ and $\Omega$; this full-column-rank condition is what makes $U_{\Theta}^T U_{\Theta}$ and $V_{\Omega}^T V_{\Omega}$ invertible, so that $\varphi^{-1}$ is well defined.
In particular, given a sketch-orthogonal factorization $Y = PSW^T$, the corresponding orthogonal factorization is $\varphi(Y) = U \Sigma V^T$ with
$$U = P (P^T P)^{-1/2}, \quad \Sigma = (P^T P)^{1/2} S (W^T W)^{1/2}, \quad V = W (W^T W)^{-1/2}.$$

\subsection{Orthogonal sketch DLRA} \label{subsec:ortho_sketch_dlra}

With the bijective map in hand, we can express the standard orthogonal projection~\eqref{eq:tangent_proj} in terms of the sketch-orthogonal bases $P$ and $W$. Since $\Range(Y) = \Range(P) = \Range(U)$ and $\Range(Y^T) = \Range(W) = \Range(V)$, we have $P(P^TP)^{-1}P^T = UU^T$ and $W(W^TW)^{-1}W^T = VV^T$, so that
\begin{equation} \label{eq:orthogonal_proj_sketch}
    \mathcal P_{Y}Z = P (P^TP)^{-1} P^T Z - P (P^TP)^{-1} P^T Z W (W^T W)^{-1} W^T +  Z W (W^T W)^{-1} W^T.
\end{equation}
This is \emph{exactly} the same projection as~\eqref{eq:tangent_proj}, only rewritten in terms of the sketch-orthogonal factors: the sketch matrices $\Theta$ and $\Omega$ enter only implicitly through the parameterization.

We then define the orthogonal sketch DLRA as the differential equation
\begin{equation} \label{eq:ortho_sketch_dlra}
\dt{Y}(t) = \mathcal P_{Y(t)} F(Y(t)), \quad Y(0) = P_0 S_0 W_0^T \in \Mr^{\Theta, \Omega},
\end{equation}
where $\mathcal P_Y$ is the \emph{standard $\ell_2$-orthogonal} projection~\eqref{eq:orthogonal_proj_sketch}. The sketch-orthogonality of $P$ and $W$ is maintained via RGS during time stepping, but the dynamics are governed by orthogonal projections. Equation~\eqref{eq:ortho_sketch_dlra} therefore coincides with the classical DLRA equation~\eqref{eq:dlra} up to a reparameterization via $\varphi$, so the classical DLRA error bound of \citet{koch2007dynamical} applies verbatim, without any sketch-induced constant.

A remark on the gauge is in order. Since $\T_Y\Mr^{\Theta,\Omega}=\T_Y\Mr$ at each $Y\in\Mr^{\Theta,\Omega}$ (Section~\ref{subsec:bijective_map}), the trajectory $Y(t)$ of~\eqref{eq:ortho_sketch_dlra} is determined only by the orthogonal projection $\mathcal P_Y$; a gauge choice only fixes how a tangent vector $\dt Y$ is split into the factor velocities $(\dt P,\dt S,\dt W)$ and leaves $Y(t)$ unchanged. We therefore adopt in Proposition~\ref{prop:ortho_sketch_coupled} the Euclidean gauge $P^T\dt P=0$, $W^T\dt W=0$ --- rather than the sketch gauge $\dt P^T\Theta^T\Theta P=0$, $\dt W^T\Omega^T\Omega W=0$ built into $\T_Y\Mr^{\Theta,\Omega}$ --- precisely because it reduces the coupled equations~\eqref{eq:ortho_sketch_coupled} to the classical form~\eqref{eq:coupled_equations}. This gauge does not keep $P,W$ sketch-orthogonal along the continuous flow, but that is immaterial: the integrators of Section~\ref{sec:integrators} re-orthogonalize the bases by RGS at every macro step, restoring the sketch-orthogonal parameterization at the discrete points where it is actually used.

\begin{proposition}[Orthogonal sketch DLRA coupled equations] \label{prop:ortho_sketch_coupled}
Let $Y(t) = P(t) S(t) W(t)^T$ be a solution of~\eqref{eq:ortho_sketch_dlra}, with factors $P(t)$, $W(t)$ of full column rank and sketch-orthogonal initial factors $P(0) \in \mathrm{St}_{\Theta}(m, r)$, $W(0) \in \mathrm{St}_{\Omega}(n, r)$. Under the gauge conditions $P^T \dt P = 0$ and $W^T \dt W = 0$, the factors $P$, $S$, $W$ satisfy the coupled equations
\begin{equation} \label{eq:ortho_sketch_coupled}
\left\{
\begin{aligned}
&\dt P = (I - P(P^TP)^{-1}P^T) F(PSW^T) W(W^TW)^{-1} S^{-1}, \\
&\dt S = (P^TP)^{-1} P^T F(PSW^T) W (W^TW)^{-1}, \\
&\dt W = (I - W(W^TW)^{-1}W^T) F(PSW^T)^T P(P^TP)^{-1} S^{-T}.
\end{aligned}
\right.
\end{equation}
\end{proposition}

\begin{proof}
Apply the map $\varphi$ of Section~\ref{subsec:bijective_map} --- whose defining formulas are well defined for any factors of full column rank, sketch-orthogonal or not --- to write $Y = U \Sigma V^T$ with
$$U = P(P^TP)^{-1/2}, \quad \Sigma = (P^TP)^{1/2} S (W^TW)^{1/2}, \quad V = W(W^TW)^{-1/2}.$$
The gauge condition $P^T \dt P = 0$ implies $\frac{d}{dt}(P^TP) = \dt P^T P + P^T \dt P = 0$, so $(P^TP)^{1/2}$ is constant along the trajectory and $\dt U = \dt P (P^TP)^{-1/2}$. Similarly, $W^T \dt W = 0$ gives $\dt V = \dt W (W^TW)^{-1/2}$ and
$\dt \Sigma = (P^TP)^{1/2} \dt S (W^TW)^{1/2}$.
Using the identities $UU^T = P(P^TP)^{-1}P^T$ and $VV^T = W(W^TW)^{-1}W^T$, the classical DLRA coupled equations~\eqref{eq:coupled_equations} become
\begin{align*}
\dt P (P^TP)^{-1/2} &= (I - P(P^TP)^{-1}P^T)\, F(PSW^T)\, W(W^TW)^{-1/2}\, \Sigma^{-1}, \\
(P^TP)^{1/2} \dt S (W^TW)^{1/2} &= (P^TP)^{-1/2} P^T F(PSW^T) W (W^TW)^{-1/2}, \\
\dt W (W^TW)^{-1/2} &= (I - W(W^TW)^{-1}W^T)\, F(PSW^T)^T\, P(P^TP)^{-1/2}\, \Sigma^{-T},
\end{align*}
and substituting $\Sigma^{-1} = (W^TW)^{-1/2} S^{-1} (P^TP)^{-1/2}$ yields~\eqref{eq:ortho_sketch_coupled} after cancellation of the $(P^TP)^{\pm 1/2}$ and $(W^TW)^{\pm 1/2}$ factors.
\end{proof}

\section{Oblique tangent space projection} \label{sec:oblique}

A natural approach to incorporate sketching into DLRA is to sketch the Galerkin condition~\eqref{eq:galerkin} that defines the approximation, in the spirit of randomized Galerkin methods for model reduction~\citep{balabanov2019randomized, balabanov2021randomized} or the Petrov--Galerkin condition in Krylov subspace methods~\citep{balabanov2022randomized, dedamas2025randomizedorthogonalizationkrylovsubspace}. For $Y = PSW^T \in \Mr^{\Theta, \Omega}$ (see Section~\ref{subsec:bijective_map}), it is natural to impose the following \emph{sketch Galerkin condition}:
\begin{align} \label{eq:sketch_galerkin_oblique}
    &\scal{\Theta (\dt Y - F(Y)) \Omega^T, \, \Theta \delta Y \Omega^T} = 0, \qquad \delta Y \in \mathcal T_{Y(t)} \mathcal M_r^{\Theta, \Omega}.
\end{align}
In this section, we show that this sketch Galerkin condition is the classical DLRA transported to the sketch space (Lemma~\ref{lem:oblique_coupled}); from it we derive the resulting oblique tangent space projection as the full-space lift, show that it is equivalent to the classical DLRA exactly when a sketched orthogonality condition holds (Proposition~\ref{prop:oblique_equivalence}), met up to the sketching error when the problem is low-rank compatible in the sense of Definition~\ref{def:lrc_approx_rk}, and analyze how the dynamics deviate otherwise.

\paragraph{Standing conventions for this section.} Throughout Section~\ref{sec:oblique} we keep $P \in \mathrm{St}_\Theta(m, r) \subset \R^{m \times r}$ and $W \in \mathrm{St}_\Omega(n, r) \subset \R^{n \times r}$, so that $\Theta P \in \R^{\ell \times r}$ and $\Omega W \in \R^{\ell \times r}$ are tall ($\ell \geq r$) with orthonormal columns. For these to be well defined, $\Theta P$ and $\Omega W$ must have full column rank $r$, which holds by Assumption~\ref{ass:embedding}. Under sketch-orthogonality, Remark~\ref{rem:pinv_notation} gives
\begin{equation} \label{eq:pinv_convention}
    (\Theta P)^+ \;=\; (\Theta P)^T \;=\; P^T \Theta^T, \qquad (\Omega W)^+ \;=\; (\Omega W)^T \;=\; W^T \Omega^T,
\end{equation}
where $^+$ denotes the Moore--Penrose pseudo-inverse (cf.~Remark~\ref{rem:pinv_notation}). Since $\Theta P$ and $\Omega W$ are tall ($\ell \geq r$) with full column rank, $(\Theta P)^+$ and $(\Omega W)^+$ are genuine \emph{left} pseudo-inverses; under sketch-orthogonality, they coincide with $(\Theta P)^T$ and $(\Omega W)^T$ respectively.

The analysis of Section~\ref{sec:oblique} is pointwise in time: each statement fixes a state $Y \in \Mr^{\Theta, \Omega}$ together with a single pair $(\Theta, \Omega)$, and concerns the oblique sketch DLRA equation at that instant. In the discrete integrator of Section~\ref{sec:integrators}, a fresh pair $(\Theta, \Omega)$ is drawn at every macro step, and the results of this section apply independently at each step.

\subsection{Oblique coupled equations} \label{subsec:oblique_coupled}

We now show that the sketch Galerkin condition~\eqref{eq:sketch_galerkin_oblique} is the classical DLRA condition transported to the sketch space: it drives the sketched state $\Theta Y \Omega^T$ by the sketched vector field $\Theta F(Y) \Omega^T$, as the next lemma makes precise.

\begin{lemma}[Oblique coupled equations] \label{lem:oblique_coupled}
Let $Y = P S W^T \in \mathcal M_r^{\Theta, \Omega}$ and set $\widehat F := \Theta\, F(P S W^T)\, \Omega^T \in \R^{\ell \times \ell}$. The sketch Galerkin condition~\eqref{eq:sketch_galerkin_oblique} uniquely determines the sketched velocities,
\begin{equation} \label{eq:oblique_sketched}
\begin{aligned}
&\dt S = (\Theta P)^T \widehat F\, (\Omega W), \\
&\Theta \dt P = \bigl(I_\ell - (\Theta P)(\Theta P)^T\bigr)\widehat F\,(\Omega W)\, S^{-1}, \\
&\Omega \dt W = \bigl(I_\ell - (\Omega W)(\Omega W)^T\bigr)\widehat F^{T}(\Theta P)\, S^{-T},
\end{aligned}
\end{equation}
which are the classical DLRA coupled equations~\eqref{eq:coupled_equations} for the sketched state $\Theta Y \Omega^T = (\Theta P) S (\Omega W)^T \in \R^{\ell \times \ell}$, with orthonormal bases $\Theta P, \Omega W$ and a vector field $\widehat F$. The factor velocities $\dt P, \dt W$ themselves are not determined by~\eqref{eq:sketch_galerkin_oblique} --- any increment with columns in $\ker\Theta$ (resp.\ $\ker\Omega$) leaves it unchanged. The lift realized by the integrators of Section~\ref{sec:integrators} replaces the sketch-space orthogonal projectors $I_\ell - (\Theta P)(\Theta P)^T$ and $I_\ell - (\Omega W)(\Omega W)^T$ by the full-space oblique projectors $I - P(\Theta P)^+\Theta$ and $I - W(\Omega W)^+\Omega$, giving $\dt Y = \dt P\, S W^T + P \dt S\, W^T + P S\, \dt W^T$ with
\begin{equation} \label{eq:oblique_coupled}
\begin{aligned}
&\dt S = (\Theta P)^+ \Theta F(P S W^T) \Omega^T (\Omega W)^{+T}, \\
&\dt P = (I - P (\Theta P)^+ \Theta) F(P S W^T) \Omega^T (\Omega W)^{+T} S^{-1}, \\
&\dt W = (I - W (\Omega W)^+ \Omega) F(P S W^T)^T \Theta^T (\Theta P)^{+T} S^{-T}.
\end{aligned}
\end{equation}
\end{lemma}
\begin{proof}
Written out, \eqref{eq:sketch_galerkin_oblique} reads $\scal{\Theta \dt Y \Omega^T - \widehat F,\ \Theta \delta Y \Omega^T} = 0$ for all $\delta Y \in \mathcal T_Y \mathcal M_r^{\Theta, \Omega}$. Since $\Theta$ and $\Omega$ have full row rank, the map $\delta Y \mapsto \Theta \delta Y \Omega^T$ sends $\mathcal T_Y \mathcal M_r^{\Theta, \Omega}$ onto the tangent space at $\widehat Y := (\Theta P) S (\Omega W)^T$ of the rank-$r$ manifold in $\R^{\ell \times \ell}$, and the sketch gauge $P^T \Theta^T \Theta \dt P = 0$ becomes the Euclidean gauge $(\Theta P)^T (\Theta \dt P) = 0$, and similarly $(\Omega W)^T (\Omega \dt W) = 0$ --- the gauge keeps $\Theta P$ and $\Omega W$ orthonormal. Hence~\eqref{eq:sketch_galerkin_oblique} is the classical Galerkin condition~\eqref{eq:galerkin} for $\widehat Y$ driven by $\widehat F$, and the derivation of~\eqref{eq:coupled_equations} yields~\eqref{eq:oblique_sketched}. Applying $\Theta$ and $\Omega$ to~\eqref{eq:oblique_coupled} returns~\eqref{eq:oblique_sketched} (using $(\Theta P)^+ = (\Theta P)^T$ and $(\Omega W)^{+T} = \Omega W$ under sketch-orthogonality), so~\eqref{eq:oblique_coupled} is a lift of the determined velocities; it is the one with $\dt K := \dt P\, S + P \dt S = F(K W^T) \Omega^T (\Omega W)^{+T}$, i.e.\ the K/L substeps of Section~\ref{sec:integrators}.
\end{proof}

\begin{lemma}[Oblique tangent space projection] \label{lem:oblique_projection}
Let $Y = P S W^T \in \Mr^{\Theta, \Omega}$, and introduce the two oblique projectors
$$
P_{\Im(Y)}^{\Theta} \;:=\; P\,(\Theta P)^+\,\Theta, \qquad P_{\Im(Y^T)}^{\Omega} \;:=\; W\,(\Omega W)^+\,\Omega,
$$
onto $\Im(Y)$ along $\mathrm{Ker}\big((\Theta P)^+ \Theta\big)$ and onto $\Im(Y^T)$ along $\mathrm{Ker}\big((\Omega W)^+ \Omega\big)$, respectively. The oblique projection onto the tangent space $\mathcal T_Y \mathcal M_r^{\Theta, \Omega}$ associated with these one-sided projectors, denoted $\mathcal P_Y^{\Theta, \Omega}$, is given by
$$\mathcal P_Y^{\Theta, \Omega} Z \;=\; P_{\Im(Y)}^{\Theta}\, Z \,+\, Z\, (P_{\Im(Y^T)}^{\Omega})^T \,-\, P_{\Im(Y)}^{\Theta}\, Z\, (P_{\Im(Y^T)}^{\Omega})^T,$$
or equivalently,
$$\mathcal P_Y^{\Theta, \Omega} Z = P P^T \Theta^T \Theta Z - P P^T \Theta^T \Theta Z \Omega^T \Omega W W^T + Z \Omega^T \Omega W W^T.$$
\end{lemma}

\begin{proof}
Substituting the coupled equations of Lemma~\ref{lem:oblique_coupled} into the identity $\mathcal P_Y^{\Theta, \Omega} F(P S W^T) = \dt P\, S W^T + P \dt S\, W^T + P S\, \dt W^T$ and extending to all $Z \in \Rmn$ by linearity yields the first expression; the compact form follows by factoring $P_{\Im(Y)}^{\Theta}$ on the left and $P_{\Im(Y^T)}^{\Omega}$ on the right. The second identity follows directly from Remark~\ref{rem:pinv_notation}.
\end{proof}

In sketch coordinates, this oblique projection is the ordinary orthogonal one: for every $Z \in \Rmn$,
$$\Theta\,(\mathcal P_Y^{\Theta, \Omega} Z)\,\Omega^T = \mathcal P_{\Theta Y \Omega^T}\!\left(\Theta Z \Omega^T\right),$$
where $\mathcal P_{\Theta Y \Omega^T}$ denotes the standard orthogonal tangent-space projection at the rank-$r$ matrix $\Theta Y \Omega^T \in \R^{\ell \times \ell}$. This is the projection-level form of Lemma~\ref{lem:oblique_coupled}.

In the rest of the paper, the \emph{oblique sketch DLRA} refers to the projected differential equation $\dt Y(t) = \mathcal P_{Y(t)}^{\Theta, \Omega} F(Y(t))$ with $Y(0) = Y_0 \in \Mr^{\Theta, \Omega}$ --- that is, the sketched Galerkin condition~\eqref{eq:sketch_galerkin_oblique} closed by the canonical lift of Lemma~\ref{lem:oblique_coupled}.

\begin{remark}[One-sided sketching] \label{rem:one_sided}
The framework covers \emph{one-sided sketching} as the special case $\Theta = I_m$ (row-only sketching) or $\Omega = I_n$ (column-only sketching). With $\Theta = I_m$, the sketch-orthogonality condition $(\Theta P)^T (\Theta P) = I_r$ collapses to $P^T P = I_r$, so $P$ becomes orthonormal in the classical sense and $(\Theta P)^+ = P^+ = P^T$; the identity in Remark~\ref{rem:pinv_notation} and all derivations of Section~\ref{sec:oblique} remain valid verbatim because they use only that $\Theta P$ is tall with full column rank and satisfies sketch-orthogonality, both of which are trivially inherited. Equations~\eqref{eq:oblique_coupled} then specialize on the $P$-side to the ordinary orthogonal formulas, while the $W$-side remains oblique:
$$
\dt P = (I - P P^T) F(P S W^T) \Omega^T (\Omega W)^{+T} S^{-1}, \quad
\dt W = (I - W (\Omega W)^+ \Omega) F(P S W^T)^T P\, S^{-T},
$$
$$
\dt S = P^T F(P S W^T) \Omega^T (\Omega W)^{+T},
$$
and symmetrically for $\Omega = I_n$. Accordingly, the oblique projector on the unsketched side reduces to the standard orthogonal projector ($P^\Theta = U U^T$ when $\Theta = I_m$, or $P^\Omega = V V^T$ when $\Omega = I_n$), so the equivalence-with-DLRA condition of Proposition~\ref{prop:oblique_equivalence} drops the OSE requirement on the unsketched side and is imposed only on the remaining sketch --- which must still embed the column or row spaces of the relevant tangent and residual matrices (cf.\ Proposition~\ref{prop:oblique_equivalence} and~\eqref{eq:oblique_residual}), not merely $\Range(P)$ or $\Range(W)$. One-sided sketching therefore interpolates between the orthogonal sketch DLRA of Section~\ref{sec:orthogonal} (when both $\Theta = I_m$ and $\Omega = I_n$) and the two-sided oblique sketch DLRA: it corresponds unconditionally to the classical DLRA on the unsketched side, and inherits the LRC- and OSE-dependent behavior of Section~\ref{subsec:equivalence} on the sketched side.
\end{remark}

\subsection{Equivalence with the classical DLRA in the low-rank compatible regime} \label{subsec:equivalence}

We now discuss the conditions under which the oblique sketch DLRA is (approximately) equivalent to the classical DLRA, and what happens when these conditions fail.

\begin{proposition} \label{prop:oblique_equivalence}
Let $\Theta \in \R^{\ell \times m}$ and $\Omega \in \R^{\ell \times n}$ be two realizations of $\mathrm{OSE}(2r, \epsilon, \delta)$ embedding the column and row spaces of $Z := \dt Y(t) - \proj{Y(t)}{F(Y(t))}$. Then solving the oblique sketch DLRA equation
$$\dt{Y}(t) = \mathcal P_{Y(t)}^{\Theta, \Omega} F(Y(t)), \qquad Y(0) = Y_0 \in \Mr^{\Theta, \Omega}$$
is equivalent to solving the DLRA if and only if
\begin{equation} \label{eq:condition_oblique_good}
\scal{\Theta \projperp{Y(t)}{F(Y(t))} \Omega^T, \Theta \delta Y(t) \Omega^T} = 0 \qquad \forall \delta Y(t) \in \mathcal T_{Y(t)} \Mr,
\end{equation}
where the perpendicular projection $\mathcal P_{Y}^{\perp} = I - \mathcal P_{Y}$ is defined from~\eqref{eq:tangent_proj}.
\end{proposition}

\begin{proof}
By Lemma~\ref{lem:oblique_coupled}, the oblique sketch DLRA satisfies the sketch Galerkin condition~\eqref{eq:sketch_galerkin_oblique}. Decomposing $F(Y(t)) = \proj{Y(t)}{F(Y(t))} + \projperp{Y(t)}{F(Y(t))}$ and recalling that $\T_{Y(t)} \Mr^{\Theta, \Omega} = \T_{Y(t)} \Mr$, we obtain for every $\delta Y(t) \in \T_{Y(t)} \Mr$
\begin{align*}
0 &= \scal{\Theta (\dt Y(t) - F(Y(t))) \Omega^T, \, \Theta \delta Y(t) \Omega^T} \\
  &= \scal{\Theta Z \Omega^T, \, \Theta \delta Y(t) \Omega^T} - \scal{\Theta \projperp{Y(t)}{F(Y(t))} \Omega^T, \, \Theta \delta Y(t) \Omega^T},
\end{align*}
where $Z := \dt Y(t) - \proj{Y(t)}{F(Y(t))} \in \T_{Y(t)} \Mr$ is a rank-$2r$ tangent vector. Evaluating this identity at $\delta Y(t) = Z$ and using that $\Theta$ and $\Omega$ embed the column and row spaces of $Z$ (Corollary~\ref{cor:two_sided}, $d = 2r$),
$$(1-\epsilon)^2 \norm{Z}_F^2 \;\leq\; \norm{\Theta Z \Omega^T}_F^2 \;=\; \scal{\Theta \projperp{Y(t)}{F(Y(t))} \Omega^T, \, \Theta Z \Omega^T}.$$
Hence condition~\eqref{eq:condition_oblique_good} forces $Z = 0$, while conversely $Z = 0$ collapses the identity to~\eqref{eq:condition_oblique_good}. The oblique scheme therefore coincides with the classical DLRA if and only if~\eqref{eq:condition_oblique_good} holds.
\end{proof}

\begin{remark}[Per-step vs.\ trajectory probability] The embedding hypothesis of Proposition~\ref{prop:oblique_equivalence} holds, for a fixed target subspace, with probability at least $1 - 2\delta$ per draw of $(\Theta, \Omega)$. Strictly speaking, the subspaces of $Z$ depend on the realized draw, so this per-draw estimate applies to subspaces fixed before the draw and should be read as a guideline for sizing $\ell$ rather than as a self-contained guarantee --- which is why the proposition is phrased deterministically, as a hypothesis on the realized sketches. In the discrete integrator of Section~\ref{sec:integrators}, a fresh independent pair is drawn at each of the $N$ macro steps; a union bound over the steps then requires the per-step OSE failure probability to be of order $\delta/N$ to keep the global failure probability at $O(\delta)$. Since $\delta$ enters logarithmically in the required sketch size, this only mildly increases $\ell$ in practice.
\end{remark}

We now turn to the large-residual regime: the oblique sketch DLRA then solves a perturbed equation whose residual does not vanish, and is therefore no longer a faithful approximation of the classical DLRA.

Following the decomposition used in the proof of Proposition~\ref{prop:oblique_equivalence}, we write $F(Y(t)) = \proj{Y(t)}{F(Y(t))} + \projperp{Y(t)}{F(Y(t))}$ and recast the sketch Galerkin condition~\eqref{eq:sketch_galerkin_oblique} as
\begin{equation} \label{eq:oblique_residual}
\scal{\Theta (\dt Y(t) - \proj{Y(t)}{F(Y(t))}) \Omega^T, \, \Theta \delta Y \Omega^T} = \scal{\Theta \projperp{Y(t)}{F(Y(t))} \Omega^T, \, \Theta \delta Y \Omega^T}, \quad \forall \delta Y \in \T_{Y(t)} \Mr.
\end{equation}
For each fixed tangent direction $\delta Y$, the left-hand side of~\eqref{eq:oblique_residual} pairs two tangent vectors of rank at most $2r$; whenever $\Theta$ and $\Omega$ embed their combined column and row spaces (of dimension at most $4r$), Corollary~\ref{cor:two_sided} (with $d = 4r$) shows that it reproduces the classical Galerkin pairing $\scal{\dt Y(t) - \proj{Y(t)}{F(Y(t))}, \, \delta Y}$ that drives the DLRA, up to an additive error $(2\epsilon + \epsilon^2) \norm{\dt Y(t) - \proj{Y(t)}{F(Y(t))}}_F \norm{\delta Y}_F$. The right-hand side is what we call the \emph{oblique residual}: when the problem is $(r,k)$-low-rank compatible, $\projperp{Y(t)}{F(Y(t))}$ has rank at most $k$, so for each fixed tangent direction $\delta Y$ the relevant column and row spaces have dimension at most $2r+k$ and an $\mathrm{OSE}(2r+k)$ controls this residual through the two-sided embedding (Corollary~\ref{cor:two_sided}, $d = 2r+k$), with probability $1 - 2\delta$: explicitly, since $\scal{\projperp{Y(t)}{F(Y(t))}, \, \delta Y} = 0$,
$$\bigl|\scal{\Theta \projperp{Y(t)}{F(Y(t))} \Omega^T, \, \Theta \delta Y \Omega^T}\bigr| \;\leq\; (2\epsilon + \epsilon^2)\, \norm{\projperp{Y(t)}{F(Y(t))}}_F\, \norm{\delta Y}_F.$$
Otherwise, $\projperp{Y(t)}{F(Y(t))}$ is large relative to what a sketch of size $\ell \ll m, n$ can control, so the oblique residual on the right-hand side of~\eqref{eq:oblique_residual} is in general a non-vanishing perturbation of the classical Galerkin condition. This perturbation drives the dynamics away from the classical DLRA solution and is confirmed numerically in Section~\ref{subsec:vlasov_poisson}, where the oblique approach leads to a visible energy drift on the Vlasov--Poisson equation.

\begin{remark}[Approximate low-rank compatibility]
The exact $\eta = 0$ setting is brittle: most problems of practical interest are only approximately low-rank compatible in the sense of Definition~\ref{def:lrc_approx_rk}, i.e., $(\eta, r, k)$-LRC with a small but non-zero tolerance $\eta$. In this case the oblique residual in~\eqref{eq:oblique_residual} is expected to remain of order $\eta$, so the oblique sketch DLRA should stay close to the classical DLRA; bounding it rigorously, however, requires more than the oblivious subspace embedding property (the part of the residual beyond rank $k$ is not embedded), and turning the residual control into a trajectory error bound would further require a perturbation argument for the projected ODE --- neither of which we pursue here.
\end{remark}

\subsection{Recap: oblique vs.\ orthogonal} \label{subsec:comparison}

The key differences between the two approaches are summarized in Table~\ref{tab:comparison}.

\begin{table}[htbp]
\centering
\begin{tabular}{@{}lcc@{}}
\toprule
 & \textbf{Oblique sketch DLRA} & \textbf{Orthogonal sketch DLRA} \\
\midrule
\textbf{Solves DLRA on} & Sketched problem $\Theta F \Omega^T$ & True problem $F$ \\
\textbf{What is sketched} & Full tangent projection & Factor orthogonalizations \\
\textbf{Underlying projector} & $\mathcal P_Y^{\Theta,\Omega}$ (Prop.~\ref{prop:oblique_equivalence}) & $\mathcal P_Y$ (classical) \\
\textbf{Equivalent to DLRA} & Iff~\eqref{eq:condition_oblique_good} (Prop.~\ref{prop:oblique_equivalence}) & Always (deterministic) \\
\textbf{Large perpendicular residual} & Non-vanishing residual~\eqref{eq:oblique_residual} & No residual \\
\bottomrule
\end{tabular}
\caption{Comparison of the oblique and orthogonal sketch DLRA approaches.}
\label{tab:comparison}
\end{table}
In summary, the oblique approach is cheaper but only conditionally and approximately correct, while the orthogonal approach is always correct at the price of an extra randomized Gram--Schmidt per factor step. For problems known to be approximately low-rank compatible, the oblique approach is a reasonable choice; otherwise, the orthogonal approach should be preferred.

\section{Sketch DLRA integrators} \label{sec:integrators}

Classical DLRA integrators --- projector-splitting (KSL)~\citep{lubich2014projector}, BUG~\citep{ceruti2022unconventional}, and its augmented and rank-adaptive variants~\citep{ceruti2022rank} --- are built from K/L/S substeps composed with Householder QR extractions. Each lifts to a sketch integrator by (i) replacing the substep projectors as in Table~\ref{tab:substeps} and (ii) replacing the Householder QR by the randomized Gram--Schmidt procedure of Section~\ref{subsec:sketch_ortho}. The splitting order, Galerkin transfer, rank truncation, and adaptivity are inherited unchanged. Throughout, $h > 0$ is the time step, the current iterate $Y_0 = P_0 S_0 W_0^T$ is sketch-orthogonal, and we use the dot notation $\dt{x}$ to denote the time derivative of $x$. In general, a fresh independent pair $(\Theta, \Omega)$ may be drawn at each macro step, the state being re-expressed in the new sketch-orthogonal bases by RGS before the substeps are taken; in the experiments of Section~\ref{sec:experiments} we instead hold a single pair $(\Theta, \Omega)$ fixed over the whole trajectory.

\subsection{Sketch substeps and integrators} \label{subsec:substeps}

In each substep, one factor is evolved while the others are frozen; the projector acting on the right-hand side selects the orthogonal or oblique variant. The three substep ODEs and the RGS basis update are collected in Table~\ref{tab:substeps}; both columns collapse to the classical DLRA substeps when $\Theta = I_m$ and $\Omega = I_n$. The sketch-orthogonality of $P$ and $W$ makes $(\Theta P) \in \R^{\ell \times r}$ and $(\Omega W) \in \R^{\ell \times r}$ have orthonormal columns, so their left inverses $(\Theta P)^T = P^T \Theta^T$ and $(\Omega W)^T = W^T \Omega^T$ are well-defined (Remark~\ref{rem:pinv_notation}) and the oblique formulas are unambiguous.

\begin{table}[ht]
\centering
\small
\setlength{\tabcolsep}{6pt}
\renewcommand{\arraystretch}{1.35}
\begin{tabular}{@{}lll@{}}
\toprule
\textbf{Substep} & \textbf{Orthogonal sketch DLRA} & \textbf{Oblique sketch DLRA} \\
\midrule
K  & $\dt K = F(K W^T)\, W (W^T W)^{-1}$
   & $\dt K = F(K W^T)\, \Omega^T (\Omega W)^{+T}$ \\
L  & $\dt L = F(P L^T)^T\, P (P^T P)^{-1}$
   & $\dt L = F(P L^T)^T\, \Theta^T (\Theta P)^{+T}$ \\
S  & $\dt S = (P^T P)^{-1} P^T F(P S W^T) W (W^T W)^{-1}$
   & $\dt S = (\Theta P)^+ \Theta \, F(P S W^T) \, \Omega^T (\Omega W)^{+T}$ \\
\midrule
Basis from $K$ & $P = \RGS(K, \Theta)$, \;\; $S \leftarrow (P^T P)^{-1} P^T K$
               & $P = \RGS(K, \Theta)$, \;\; $S \leftarrow (\Theta P)^+ \Theta K$ \\
Basis from $L$ & $W = \RGS(L, \Omega)$, \;\; $S \leftarrow (W^T W)^{-1} W^T L$
               & $W = \RGS(L, \Omega)$, \;\; $S \leftarrow (\Omega W)^+ \Omega L$ \\
\bottomrule
\end{tabular}
\caption{Substeps of the sketch DLRA in orthogonal and oblique form. Bases $P \in \Rmr$ and $W \in \R^{n \times r}$ are sketch-orthogonal, $(\Theta P)^T (\Theta P) = I_r$ and $(\Omega W)^T (\Omega W) = I_r$. The classical DLRA substeps are recovered in both columns when $\Theta = I_m$ and $\Omega = I_n$, so that $(\Omega W)^+ \Omega = W^T$ and $(\Theta P)^+ \Theta = P^T$ collapse to the standard orthogonal projections.}
\label{tab:substeps}
\end{table}

Every classical DLRA integrator lifts to the sketch setting by picking one column of Table~\ref{tab:substeps} and using RGS for the basis extraction. Below, we report the two variants used in the experiments. The same recipe extends straightforwardly to the augmented BUG and rank-adaptive BUG integrators of~\citep{ceruti2022rank}, which we do not describe here.

\paragraph{Sketch projector-splitting integrator (sKSL).}
A Lie--Trotter step performs the sequence K, $-$S, L on $[0, h]$: the K-substep starts from $K(0) = P_0 S_0$ and produces $P_1 = \RGS(K(h), \Theta)$; the S-substep uses the \emph{negated} right-hand side; the L-substep ends with $W_1 = \RGS(L(h), \Omega)$~\citep{lubich2014projector}. A symmetric Strang splitting K/2, $-$S/2, L, $-$S/2, K/2 gives second order at the cost of one extra S-substep.

\paragraph{Sketch BUG integrator (sBUG).}
The K- and L-substeps are solved \emph{in parallel} from $(P_0, S_0, W_0)$, yielding $P_1 = \RGS(K(h), \Theta)$ and $W_1 = \RGS(L(h), \Omega)$~\citep{ceruti2022unconventional}. The S-substep is then run from the Galerkin transfer $S_{\mathrm{init}} = M S_0 N^T$ with
$$M = (P_1^T P_1)^{-1} P_1^T P_0, \quad N = (W_1^T W_1)^{-1} W_1^T W_0 \qquad \text{(orthogonal)},$$
$$M = (\Theta P_1)^+ (\Theta P_0), \quad N = (\Omega W_1)^+ (\Omega W_0) \qquad \text{(oblique)}.$$

\begin{remark}[Parallelism] \label{rem:parallelism}
In sBUG the K- and L-substeps depend only on $(P_0, S_0, W_0)$ and can be dispatched to separate devices. With RGS, each branch reduces to a tall-and-skinny sketch multiplication $\Theta K$ (Level-3 BLAS) followed by a small $\ell \times r$ QR, making it well-suited for GPUs; the $r \times r$ S-substep is negligible. The sKSL integrator is inherently sequential because its S-substep needs $P_1$ before L can start.
\end{remark}

\subsection{Advantages of sketch DLRA on modern hardware} \label{subsec:cost}

At leading order, the two integrators have the same basis-update flop count, $O(m r^2)$ vs.\ $O(m \ell r)$ with $\ell = O(r)$, so the practical speedup of the sketch approach is not explained by flop savings alone. It comes from three additional mechanisms that flop counts miss~\citep{demmel2012communication, balabanov2022randomized, higgins2025analysis}: (i) fewer \emph{global synchronizations} per basis update, (ii) suitability for exploiting mixed precision while maintaining numerical stability, and (iii) a shift from bandwidth-bound BLAS-2 panel factorizations to compute-bound BLAS-3 matrix multiplications for well-conditioned bases. We therefore report per-step cost along three axes: flops, global synchronizations among processors, and BLAS level of the dominant kernel. Throughout, the tall factors $K \in \Rmr$ and $L \in \R^{n \times r}$ are assumed row-distributed across processors, and $F$ is assumed to factorize through the low-rank structure so that the cost of computing $F(P S W^T) W$ and $F(P S W^T)^T P$ stays contained, as for the test problems of Section~\ref{sec:experiments}.

\paragraph{Standard basis update.} Different algorithms can be used to compute an orthogonal basis, such as Householder QR, classical Gram--Schmidt, or modified Gram--Schmidt. For a row-distributed $K$, the column-by-column Householder QR needs two global synchronizations per column to compute the factorization~\citep{demmel2012communication} and $r$ synchronizations to explicitly compute the $Q$ factor from its compact representation, for a total of $3r$ synchronizations. The dominant kernel is a BLAS-2 operation on the full row-dimension $m$. The flop count is $4 m r^2$, higher than classical and modified Gram--Schmidt, as the basis needs to be explicitly constructed from its Householder-based representation. Communication-avoiding variants such as TSQR and CAQR collapse this to a single AllReduce, but require a bespoke reduction tree on Householder factors that most high-level scientific-computing stacks do not expose. Modified Gram--Schmidt requires $j-1$ synchronizations to compute column $j$ for a total of $O(r^2)$ synchronizations: column $j$ cannot start until the inner products and norm computed from column $j-1$ have been combined between processors. Although classical Gram--Schmidt requires $r$ synchronizations, it is known to be unstable in practice. Both classical and modified Gram--Schmidt require $2m r^2$ flops.

\paragraph{Sketch basis update.} RGS has the same communication cost as classical Gram--Schmidt and, on the long-vector updates, half the flops: $r$ synchronizations and $mr^2$ flops, to which the sketch step adds $O(mr\log\ell)$ for an SRHT sketch (sub-leading) or $O(m\ell r) \approx 2mr^2$ for a dense Gaussian sketch (same order as Householder QR). However, it is shown to be numerically as stable as modified Gram--Schmidt and is also suitable for using mixed precision~\citep{balabanov2022randomized}. When $K$ is well conditioned and randomized Cholesky QR can be used, the column-by-column process is replaced by a single sketch of the whole tall factor followed by a small QR:
\begin{enumerate}
    \item $\Theta K \in \R^{\ell \times r}$ --- one BLAS-3 tall-and-skinny matmul; one global reduction of an $\ell \times r$ block.
    \item Thin QR of $\Theta K$ --- $O(\ell r^2)$ flops, entirely local.
    \item $P = K R^{-1}$ --- BLAS-3 matmul, embarrassingly parallel on the row-distributed $K$.
\end{enumerate}
The total per-basis synchronization count is $1$, independent of $r$, and every kernel is BLAS-3. Step~1 costs $O(m \ell r)$ flops for a dense Gaussian sketch; subsampled randomized Hadamard transforms reduce this to $O(m r \log \ell)$~\citep{ailon2006approximate, balabanov2022randomized}. On GPUs, where BLAS-2 panel kernels are memory-bound, and on TPUs and tensor cores, which expose only batched matrix products, this BLAS-3 rCholQR path is the precondition for the basis update to approach peak hardware throughput.

\paragraph{Per-step totals.} The sKSL integrator runs the K- and L-basis updates sequentially (the S-substep needs $P_1$ before L starts), so its per-step synchronization count is $O(r^2)$ for modified Gram--Schmidt, $O(r)$ for Householder QR, classical Gram--Schmidt, and RGS; and $O(1)$ when randomized Cholesky QR is used in place of RGS. The sBUG variant dispatches the K- and L-updates to separate devices (Remark~\ref{rem:parallelism}), which halves both counts on two devices; the rCholQR total stays $O(1)$ independently of $r$. On distributed or multi-GPU hardware where reduction latency is a material fraction of the step cost, this is a real advantage at the ranks $r \in [10, 100]$ typical of DLRA applications~\citep{balabanov2022randomized, higgins2025analysis}.

As a disclaimer, our experiments (Section~\ref{sec:experiments}) run on a single CPU on \texttt{numpy}/\texttt{scipy} at small problem sizes where reduction latency is negligible; they validate the theory of Sections~\ref{sec:orthogonal} and~\ref{sec:oblique} rather than measure wall-clock speedup. A speedup is expected only on large problems with $m, n$ in the range $10^5$--$10^7$, especially for the parallel BUG and augmented-BUG variants, or in a higher-dimensional tensor setting (see future directions in the conclusion section).

The bases $P$ and $W$ produced by RGS with $(r, \epsilon, \delta)$-OSE sketches are also nearly Euclidean-orthonormal --- $\tfrac{1}{1+\epsilon} I \preceq P^T P,\, W^T W \preceq \tfrac{1}{1-\epsilon} I$ with probability at least $1-\delta$ per basis ($1-2\delta$ jointly, $\Theta$ and $\Omega$ being independent) --- so the inverses $(P^T P)^{-1}$ and $(W^T W)^{-1}$ that appear in the orthogonal substeps of Table~\ref{tab:substeps} are well-conditioned and can be computed stably via Cholesky; see~\citep{balabanov2022randomized} for the detailed RGS stability analysis.

\section{Numerical experiments} \label{sec:experiments}

We compare three families of integrators on three test problems that span the spectrum of low-rank compatibility. The integrators are: (i) the classical DLRA BUG method of \citet{ceruti2022unconventional} as an unsketched baseline; (ii) the \emph{oblique} sketch DLRA variants of KSL and BUG analyzed in Section~\ref{sec:oblique}; and (iii) the \emph{orthogonal} sketch DLRA variants of KSL and BUG of Section~\ref{sec:orthogonal}. All experiments are run on a MacBook Pro with an Apple M1 chip and $16\,$GB of RAM. The implementation is in Python and is publicly available on GitHub\footnote{\url{https://github.com/BenjaminCarrel/Sketch-DLRA}}. The sketch matrices are Gaussian $\Theta \in \R^{\ell \times m}$ and $\Omega \in \R^{\ell \times n}$, with sketch size $\ell$ reported for each experiment. All sketch-based methods produce their sketch-orthogonal bases via the randomized Gram--Schmidt (RGS) process of~\citet{balabanov2022randomized}, followed by a single randomized Cholesky QR refinement to enforce sketch-orthogonality up to machine precision \citep{dedamas2025arnoldi} (Section~\ref{subsec:sketch_ortho}). A single pair $(\Theta, \Omega)$ is drawn and held fixed over the whole trajectory (rather than resampled at each macro step as in Section~\ref{sec:integrators}), and the statistics below are taken over independent such draws. For each sketch method, the reported curve is the median relative Frobenius error
$$
\mathrm{err}(t) \;:=\; \frac{\norm{Y(t) - A_{\mathrm{ref}}(t)}_F}{\norm{A_{\mathrm{ref}}(t)}_F}
$$
over $20$ independent draws of $(\Theta, \Omega)$ ($5$ for Vlasov--Poisson, see Section~\ref{subsec:vlasov_poisson}); the shaded band is the interquartile range. The reference trajectory $A_{\mathrm{ref}}(t)$ is obtained by integrating the full-order problem~\eqref{eq:fom_intro} with \texttt{scipy}'s adaptive RK45 at tight tolerance. As a benchmark for the best achievable rank-$r$ approximation we also plot $\norm{\sigma_r(A_{\mathrm{ref}}(t))}_F / \norm{A_{\mathrm{ref}}(t)}_F$, the relative error of the best rank-$r$ SVD truncation. The experiments are organized from the most favorable case (Allen--Cahn, approximately LRC) to the most adversarial case (Vlasov--Poisson, large perpendicular residual).

\subsection{Allen--Cahn equation} \label{subsec:allen_cahn}

The Allen--Cahn equation~\citep{allen1979microscopic} describes the process of phase separation in multi-component alloy systems.
In its simplest form, it is expressed as the partial differential equation
\begin{equation} \label{eq:allen_cahn}
\begin{aligned}
&\dt{u}(t, x, y) = \varepsilon \Delta u(t, x, y) + u(t, x, y) - u(t, x, y)^3, &&(x,y) \in \mathcal{D} \subset \R^2, \quad t \in [0, T],  \\
&u \text{ periodic on } \partial \mathcal{D}, &&t \in [0, T], \\
&u(0, x, y) = u_0(x, y) &&(x, y) \in \mathcal{D},
\end{aligned}
\end{equation}
where $\Delta u$ is the diffusion term, $u - u^3$ is the reaction term, and $\varepsilon$ is a small parameter that controls the stiffness of the equation.
By discretizing~\eqref{eq:allen_cahn} using second-order centered finite differences on a uniform tensor-product grid with $n$ points per direction, we obtain the Sylvester-like matrix differential equation
\begin{equation} \label{eq:allen_cahn_discrete}
\dt{X}(t) \;=\; A\, X(t) + X(t)\, A^T + X(t) - X(t)^{*3},
\end{equation}
where $A \in \R^{n \times n}$ is $\varepsilon$ times the symmetric one-dimensional finite-difference Laplacian and the power is taken element-wise ($X^{*3} = X * X * X$, with $*$ the Hadamard product). The diffusion $A X + X A^T$ preserves rank-$r$ structure exactly, and for smooth initial data the reaction $X - X^{*3}$ acts as a low-rank perturbation along the trajectory. This places Allen--Cahn squarely in the approximately $(r, k)$-low-rank compatible regime for a small $k$, and the low-rank-compatibility analysis of Section~\ref{subsec:equivalence} predicts that the oblique sketch DLRA should perform comparably to its orthogonal counterpart on this approximately-LRC problem.

We run~\eqref{eq:allen_cahn_discrete} on $\mathcal{D} = [0, 2\pi]^2$ with $n = 64$, $\varepsilon = 0.01$, rank $r = 8$, time horizon $t \in [0, 10]$, a DLRA macro step $\Delta t = 1/60$, and a solution snapshot saved every $10$ DLRA steps (i.e., every $1/6$ time units, for $61$ snapshots in total). Within each DLRA step, the K, L, and S sub-problems are solved with \texttt{scipy}'s adaptive RK45 at relative and absolute tolerance $10^{-12}$. Figure~\ref{fig:allen_cahn} shows the relative error of the oblique and orthogonal sketch variants at sketch size $\ell = 2r = 16$, compared against the unsketched BUG baseline and the best rank-$r$ SVD truncation of the reference trajectory. Both sketch families produce sensible low-rank approximations of the reference solution. The oblique variant tracks the orthogonal one in the median, confirming the LRC prediction --- on a problem where the perpendicular residual is near low-rank, the equivalence regime of Section~\ref{subsec:equivalence} applies --- but with a visibly wider interquartile band, the first sign of the run-to-run sensitivity that becomes severe on the Vlasov--Poisson problem, where the perpendicular residual is large.

\begin{figure}[htbp]
    \centering
    \includegraphics[width=0.8\textwidth]{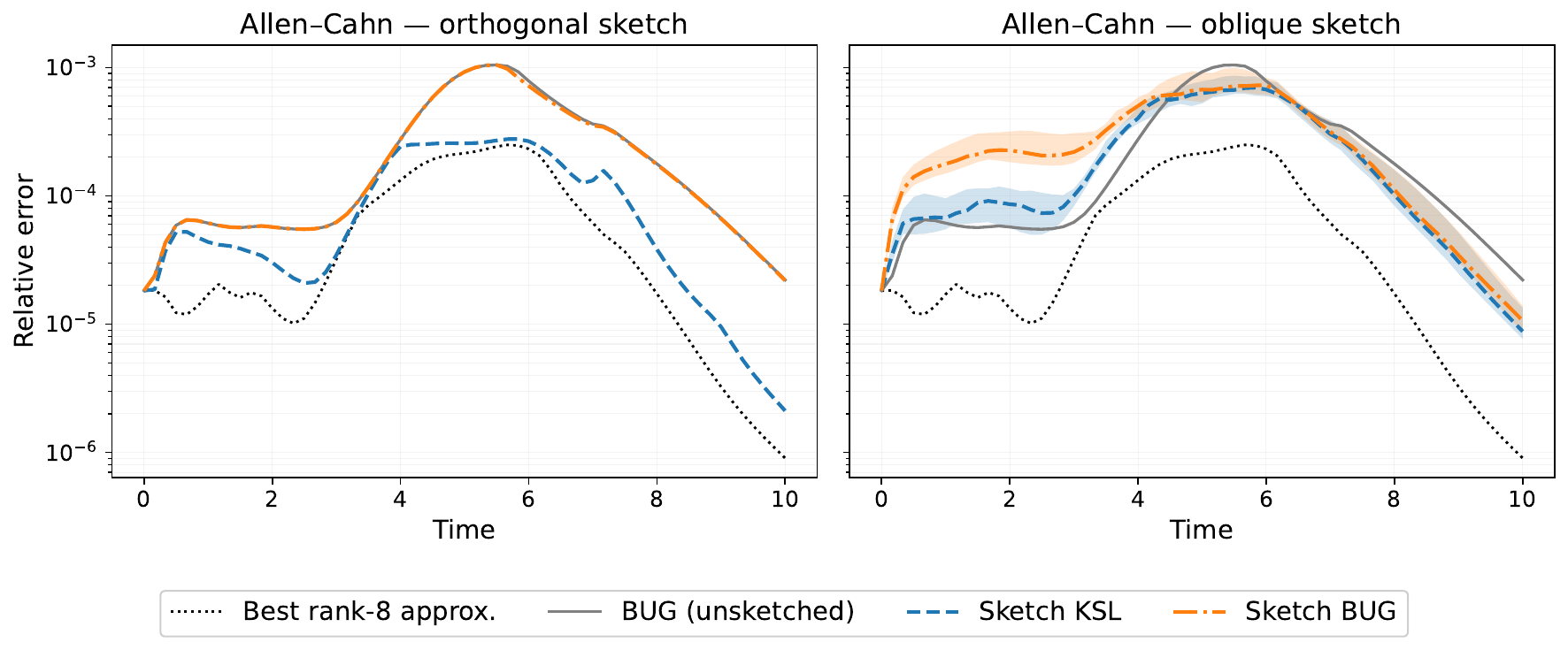}
    \caption{Allen--Cahn equation~\eqref{eq:allen_cahn_discrete} with $n = 64$, $r = 8$, $\varepsilon = 0.01$, $\mathcal{D} = [0, 2\pi]^2$, sketch size $\ell = 2r = 16$. Each sketch curve is the median over $20$ Gaussian sketches; the shaded band is the interquartile range. The solid gray curve is the unsketched BUG baseline and the black dotted curve is the best rank-$r$ SVD truncation of the reference. Left: orthogonal sketch variants. Right: oblique sketch variants.}
    \label{fig:allen_cahn}
\end{figure}

\subsection{Fokker--Planck equation} \label{subsec:fokker_planck}

The Fokker--Planck equation describes the evolution of the probability density function (PDF) of the $d$-dimensional state $X_t$ of the It\^o stochastic differential equation
$$
d X_t \;=\; \mu(X_t)\, dt + \sigma\, dW_t,
$$
where $\mu : \R^d \to \R^d$ is the drift, $\sigma > 0$ is a constant diffusion coefficient, and $W_t$ is the $d$-dimensional standard Wiener process. The associated PDF $f(t, x)$ satisfies
\begin{equation} \label{eq:fokker_planck}
    \partial_t f(t, x) \;=\; - \sum_{i=1}^d \partial_{x_i}\!\bigl(\mu_i(x)\, f(t, x)\bigr) + \frac{\sigma^2}{2} \sum_{i=1}^d \partial^2_{x_i} f(t, x),
\end{equation}
with initial condition $f(0, x) = f_0(x)$.

Low-rank methods for the Fokker--Planck PDF are a natural setting: closely related are the dynamically orthogonal (DO) field equations~\citep{sapsis2009dynamically, musharbash2015error} and the rank-adaptive tensor approach of~\citep{dektor2021rank}, whose drift we adopt below. Following~\citep{dektor2021rank}, we take $\sigma = 2$ and the drift
$$
\mu_i(x) \;=\; \bigl(\gamma(x_{i+1}) - \gamma(x_{i-2})\bigr)\, \xi(x_{i-1}) - \phi(x_i),
\qquad i = 1, \ldots, d,
$$
with indices interpreted modulo $d$ and $\gamma, \xi, \phi$ the $2\pi$-periodic functions $\gamma(x) = \sin(x)$, $\xi(x) = \cos(x)$, $\phi(x) = \exp(\sin(x)) + 1$. We specialize to $d = 2$ on the flat torus $\mathcal{D} = [0, 2\pi]^2$ with the initial condition
$$
f_0(x_1, x_2) \;=\; \frac{1}{m_0}\Bigl(e^{\sin(x_1 - x_2)^2} + \sin(x_1 + x_2)^2\Bigr),
$$
where $m_0$ normalizes $f_0$ to unit Frobenius norm on the discretization grid.

Discretizing~\eqref{eq:fokker_planck} with $n$ equispaced points per direction and centered finite differences, we obtain the Sylvester-like matrix differential equation
\begin{equation} \label{eq:fokker_planck_discrete}
    \dt{X}(t) \;=\; D_{xx}\, X(t) + X(t)\, D_{yy}^T - D_x\bigl(M_1 * X(t)\bigr) - \bigl(M_2 * X(t)\bigr)\, D_y^T, \qquad X(0) = X_0,
\end{equation}
where $D_{xx}, D_{yy} \in \R^{n \times n}$ are the scaled periodic 1D Laplacians, $D_x, D_y \in \R^{n \times n}$ are periodic centered first-difference operators, $*$ is the Hadamard product, and $M_1, M_2 \in \R^{n \times n}$ are the drift coefficients $\mu_1, \mu_2$ evaluated on the grid.

Low-rank compatibility is less clear-cut here than for Allen--Cahn: the diffusion $D_{xx} X + X D_{yy}^T$ is rank-preserving, but the drift matrices $M_1, M_2$ are not separable in $(x, y)$, so the Hadamard products $M_1 * X$ and $M_2 * X$ generically raise the rank of $X$. In practice, however, $M_1$ and $M_2$ built from the smooth periodic functions $\gamma, \xi, \phi$ are exactly rank-$3$ matrices\footnote{Expanding $\mu_1(x_1, x_2) = \sin(x_2)\cos(x_2) - \sin(x_1)\cos(x_2) - (e^{\sin(x_1)} + 1)$ exhibits $\mu_1$ as a sum of three separable functions --- a function of $x_2$ alone, a tensor product in $(x_1, x_2)$, and a function of $x_1$ alone --- so $M_1$ has rank at most $3$ on any tensor-product grid; the same holds for $M_2$ by symmetry. A numerical SVD at $n = 64$ confirms $\sigma_4(M_i) / \sigma_1(M_i) \approx 10^{-16}$ for $i = 1, 2$, so $M_1 * X$ and $M_2 * X$ stay at rank at most $3 r$ and remain approximately low-rank along the trajectory.} This places Fokker--Planck in the approximate-LRC regime, and the analysis of Section~\ref{subsec:equivalence} predicts that the oblique variant should produce sensible results, as for Allen--Cahn.

We run~\eqref{eq:fokker_planck_discrete} with $n = 64$, rank $r = 8$, time horizon $t \in [0, 1]$, a DLRA macro step $\Delta t = 1/600$, and a solution snapshot saved every $10$ DLRA steps (i.e., every $1/60$ time units, for $61$ snapshots in total). Within each DLRA step, the K, L, and S sub-problems are solved with \texttt{scipy}'s adaptive RK45 at relative and absolute tolerance $10^{-12}$. Figure~\ref{fig:fokker_planck} shows the relative error of the oblique and orthogonal sketch variants at sketch size $\ell = 2r = 16$, compared against the unsketched BUG baseline and the best rank-$r$ SVD truncation of the reference trajectory. The orthogonal sketch variants (left panel) concentrate tightly around the unsketched BUG error --- the interquartile band is essentially invisible over $20$ Gaussian sketches --- and the orthogonal KSL variant is actually slightly below BUG, reflecting the standard ordering between projector-splitting and BUG integrators rather than a sketching effect. The oblique sketch variants (right panel) also compute a good low-rank approximation, but exhibit a clearly visible interquartile band throughout the trajectory, including a wider spread on the final error than their orthogonal counterparts.

\begin{figure}[htbp]
    \centering
    \includegraphics[width=0.8\textwidth]{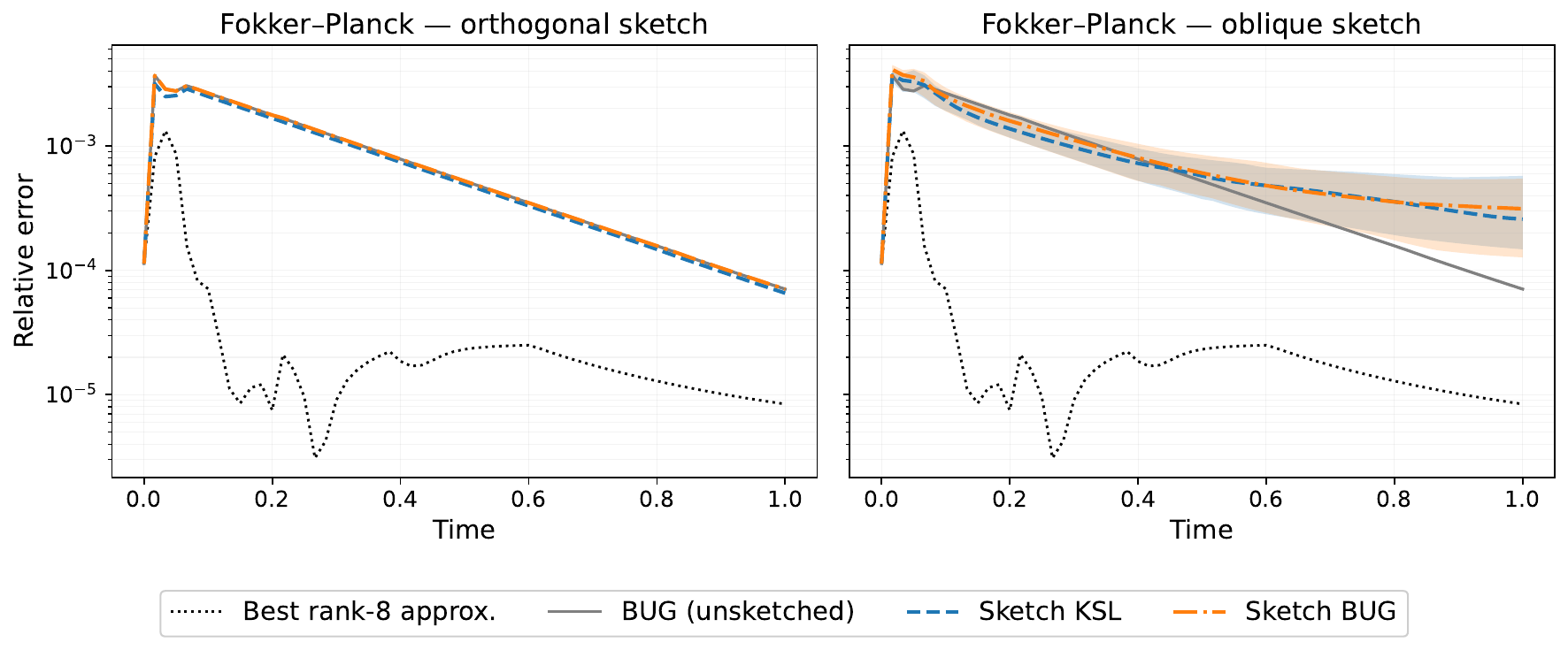}
    \caption{Fokker--Planck equation~\eqref{eq:fokker_planck_discrete} with $n = 64$, $r = 8$, $\sigma = 2$, $\mathcal{D} = [0, 2\pi]^2$, sketch size $\ell = 2r = 16$. Each sketch curve is the median over $20$ Gaussian sketches; the shaded band is the interquartile range. The solid gray curve is the unsketched BUG baseline and the black dotted curve is the best rank-$r$ SVD truncation of the reference. Left: orthogonal sketch variants. Right: oblique sketch variants.}
    \label{fig:fokker_planck}
\end{figure}

\subsection{Vlasov--Poisson equation} \label{subsec:vlasov_poisson}

The Vlasov--Poisson equation models the evolution of an electron distribution in a collisionless plasma, in the electrostatic regime with a uniform neutralizing ion background:
\begin{equation} \label{eq:vlasov_poisson}
\begin{aligned}
&\partial_t f + v \cdot \nabla_x f - E(f) \cdot \nabla_v f = 0, \\
&\nabla_x \cdot E(f) = 1 - \int f\, dv, \qquad \nabla_x \times E(f) = 0,
\end{aligned}
\end{equation}
where $f(t, x, v) \ge 0$ is the particle density and $E(f)$ is the self-consistent electric field, with periodic boundary conditions in $x$. DLRA methods for this equation have been developed in~\citep{einkemmer2018low, einkemmer2019quasi, einkemmer2021mass}. We specialize to $d = 1$ and discretize on an $n_x \times n_v$ grid using second-order centered finite differences; the unknown matrix $X(t) \in \R^{n_x \times n_v}$ satisfies
\begin{equation} \label{eq:vlasov_poisson_discrete}
    \dt{X}(t) \;=\; -\, D_x\, X(t)\, \Lambda_v \;+\; E\bigl(t, X(t)\bigr)\, X(t)\, D_v^T, \qquad X(0) = X_0,
\end{equation}
with $D_x, D_v$ the periodic centered first-difference operators, $\Lambda_v = \operatorname{diag}(v_j)$ the velocity grid, and $E(t, X) = \operatorname{diag}(E_1, \ldots, E_{n_x})$ the diagonal matrix of discrete electric-field values. In the nonlinear regime the solution is no longer well approximated on the rank-$r$ manifold, so the perpendicular component $\projperp{Y}{F(Y)}$ becomes large and the oblique residual of~\eqref{eq:oblique_residual} amplifies quickly instead of (approximately) vanishing; the equivalence of Proposition~\ref{prop:oblique_equivalence} is then lost. This is the adversarial test case that motivates orthogonal sketch DLRA over oblique sketch DLRA.

We run~\eqref{eq:vlasov_poisson_discrete} on the standard two-stream instability~\citep{cheng1976integration, filbet2003comparison} with $(\mathcal{D}_x, \mathcal{D}_v) = ([0, 10\pi], [-6, 6])$, $n_x = n_v = 64$, rank $r = 8$, time horizon $t \in [0, 60]$, and initial condition
$$
f_0(x, v) \;=\; \frac{1}{2\sqrt{2\pi}}\Bigl(e^{-(v - v_0)^2 / 2} + e^{-(v + v_0)^2 / 2}\Bigr)\bigl(1 + \alpha \cos(k x)\bigr),
\qquad v_0 = 2.4,\ k = 0.2,\ \alpha = 10^{-3},
$$
producing two counter-streaming beams perturbed by a small linear mode. We use a DLRA macro step $\Delta t = 1/20$ with \texttt{scipy} adaptive RK45 substeps at tolerance $10^{-12}$. Figure~\ref{fig:vlasov_poisson} reports the relative error and the electric energy $\tfrac{1}{2}\int |E(t, x)|^2 dx$ at sketch size $\ell = 2r = 16$, each curve a median over $5$ Gaussian draws (fewer than for Allen--Cahn and Fokker--Planck because the failure is already unambiguous; individual oblique runs trigger inner-integrator failures and yield truncated curves). The orthogonal variants (blue) track the unsketched BUG baseline within a constant factor across the full window, including through the nonlinear saturation; the oblique variants (red) diverge exponentially, with the oblique KSL variant eventually causing the inner integrator to fail.

The failure is structural. Vlasov--Poisson is Hamiltonian and conserves total energy $E_{\mathrm{kin}} + E_{\mathrm{el}}$ and total mass $\iint f\,dx\,dv$, and the classical DLRA inherits both up to the rank-$r$ truncation error; at our parameters the orthogonal sketch BUG preserves them to order $10^{-2}$ over $[0, 60]$ (Figure~\ref{fig:vlasov_poisson_conservation}). The oblique sketch DLRA, by contrast, solves a perturbed ODE $\dt{Y} = \mathcal P^{\Theta, \Omega}_{Y} F(Y)$ whose invariants depend on the random sketches rather than on the physics: total energy drifts by more than $100\%$ before $t = 40$, and total mass goes negative at $t \approx 38$. The mechanism is the one diagnosed in Section~\ref{subsec:equivalence}: in the nonlinear regime the perpendicular residual $\projperp{Y}{F(Y)}$ is large and far from low-rank, so no sketch of size $\ell \ll m, n$ can make the oblique residual of~\eqref{eq:oblique_residual} vanish, and the equivalence of Proposition~\ref{prop:oblique_equivalence} is lost. Increasing the sketch size within the low-rank budget does not rescue the oblique scheme: the failure is intrinsic to the oblique projection on this problem, not an artifact of the sketch size $\ell = 2r$. The orthogonal sketch DLRA sidesteps both issues. Its tangent projection is the standard orthogonal projector~\eqref{eq:orthogonal_proj_sketch} regardless of $(\Theta, \Omega)$ (Section~\ref{sec:orthogonal}), so the trajectory coincides with the classical DLRA solution and inherits its conservation behavior. The sketches still enter through the RGS basis extraction, but there an OSE plays only a \emph{stability} role --- keeping $P^TP$ and $W^TW$ well-conditioned in the orthogonal substeps of Table~\ref{tab:substeps} --- rather than the \emph{correctness} role it had to play in the oblique scheme through Proposition~\ref{prop:oblique_equivalence}.

\begin{figure}[htbp]
    \centering
    \includegraphics[width=0.8\textwidth]{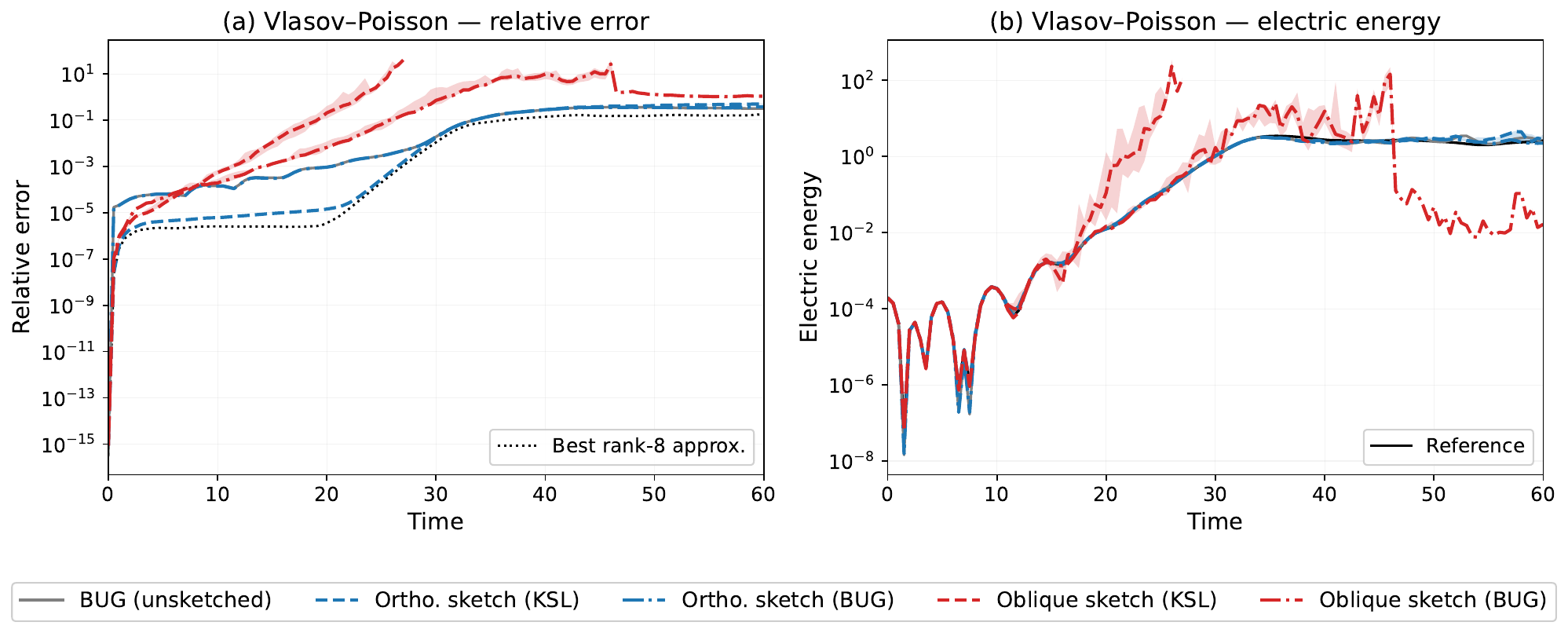}
    \caption{Vlasov--Poisson equation~\eqref{eq:vlasov_poisson_discrete} (two-stream instability) with $n_x = n_v = 64$, $r = 8$, sketch size $\ell = 2r = 16$, and $t \in [0, 60]$. Each sketch curve is the median over $5$ Gaussian sketches; the shaded band is the interquartile range. Left: relative Frobenius error. Right: electric energy $\tfrac{1}{2}\int |E(t, x)|^2 dx$. The orthogonal sketch variants (blue) track the unsketched BUG baseline (gray) on both quantities; the oblique sketch variants (red) diverge, and the electric energy quickly overshoots the reference by more than an order of magnitude before the inner integrator fails.}
    \label{fig:vlasov_poisson}
\end{figure}

\begin{figure}[htbp]
    \centering
    \includegraphics[width=\textwidth]{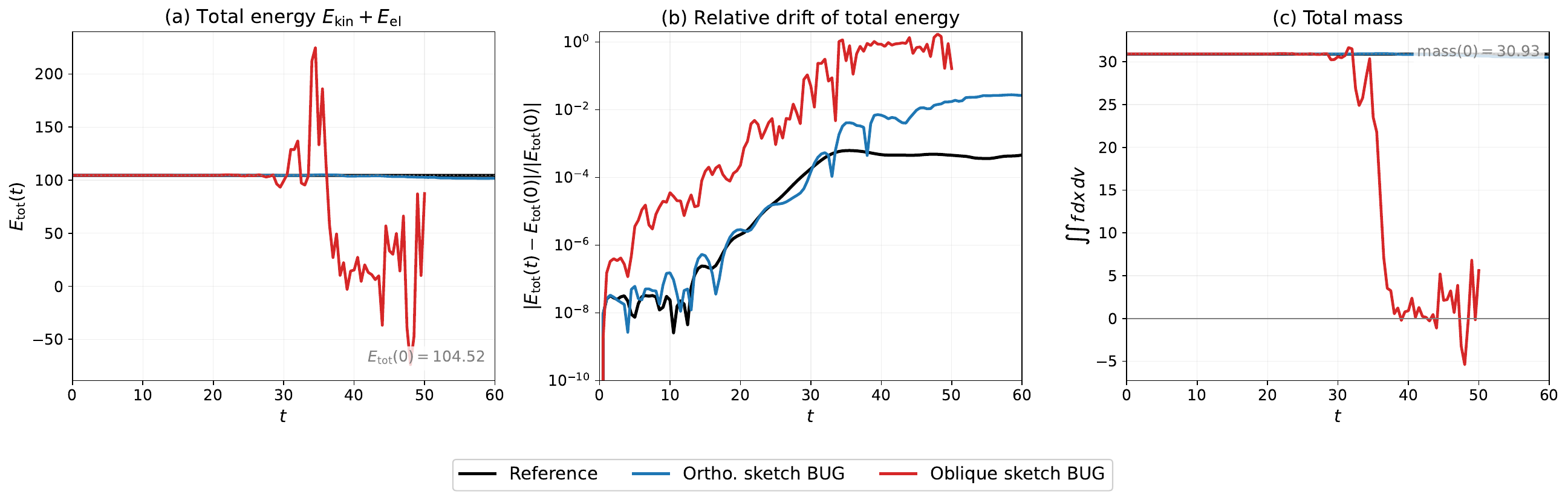}
    \caption{Conservation laws on the Vlasov--Poisson two-stream instability, same parameters as Figure~\ref{fig:vlasov_poisson}. Left: total energy $E_{\mathrm{tot}}(t) = E_{\mathrm{kin}}(t) + E_{\mathrm{el}}(t)$. Middle: relative drift $|E_{\mathrm{tot}}(t) - E_{\mathrm{tot}}(0)| / |E_{\mathrm{tot}}(0)|$, log scale. Right: total mass $\iint f\,dx\,dv$. The reference (black) and orthogonal sketch BUG (blue) preserve both invariants to order $10^{-2}$ over the full window, tracking each other; the oblique sketch BUG (red) violates total energy by more than $100\%$ before $t = 40$ and its mass goes negative at $t \approx 38$, demonstrating that the oblique sketch DLRA is solving a different ODE whose invariants are unrelated to the Vlasov--Poisson Hamiltonian.}
    \label{fig:vlasov_poisson_conservation}
\end{figure}

\section{Conclusion} \label{sec:conclusion}

The orthogonal sketch DLRA we propose applies to any problem (LRC or not) and inherits the classical DLRA error bound, while reducing synchronizations --- and, for structured sketches, flops --- on the basis-update step.
Sketching the Galerkin condition that defines the DLRA leads to an oblique tangent projection whose dynamics deviate from the classical DLRA and fail on problems with a large perpendicular residual. The orthogonal sketch DLRA instead evolves sketch-orthogonal bases under standard orthogonal projections, preserving the geometric structure of the classical DLRA and replacing the only BLAS-2 kernel of the pipeline (Householder QR) by randomized Gram--Schmidt and, when applicable, randomized Cholesky QR, which together let the basis update fully exploit GPU hardware (Section~\ref{subsec:cost}). We derived sketch versions of the projector-splitting and BUG integrators and validated them on the Allen--Cahn, Fokker--Planck, and Vlasov--Poisson equations. The contribution is exploratory: our small-CPU experiments validate the theory rather than measure wall-clock speedup, which is expected only at large $m, n$.

Three directions follow. First, subspace iteration could be considered to address cases with slowly decaying singular values, where randomized methods are typically less accurate. Second, the parallelism of the sketch BUG integrators could be exploited: the K- and L-substeps are independent, each reducing to an ODE solve and a Level-3 BLAS sketch multiplication, well-suited to a two-GPU pipeline. Third, the framework could be extended to Tucker and tensor-train formats, where the QR bottleneck is more pronounced and RGS extends naturally to each mode. The sketch integrators developed here also combine directly with the low-rank Parareal parallel-in-time integrator of \citet{carrel2023low}, adding parallelism across the time direction.

\bibliographystyle{plainnat}
\bibliography{citations}

\end{document}